\documentclass[a4paper,12pt]{amsart}

\usepackage{geometry}
\usepackage{amssymb}
\usepackage{graphicx}
\usepackage{amsmath,amsthm}
\usepackage[english]{babel}
\usepackage{tikz-cd}
\usepackage{placeins}
\usepackage{hyperref}
\usepackage{todonotes}

\newcommand{\R}{\ensuremath{\mathbb{R}}}
\newcommand{\lb}{\ensuremath{\langle}}
\newcommand{\rb}{\ensuremath{\rangle}}
\newcommand{\pder}[2]{\ensuremath{\frac{\partial #1}{\partial #2} } }
\newcommand{\np}{\ensuremath{\mathrm{np} } } 
\newcommand{\ns}{\ensuremath{\mathrm{ns} } } 
\newcommand{\db}{\ensuremath{\mathrm{db} } } 
\newcommand{\s}{\ensuremath{\mathrm{s} } }   
\newcommand{\g}{\ensuremath{\mathrm{g} } }   
\newcommand{\pr}{\ensuremath{\mathrm{pr}}}
\renewcommand{\d}{\ensuremath{\mathrm{d}}} 
\newcommand{  \T}{\ensuremath{\mathrm{T}}} 
\newcommand{  \D}{\ensuremath{\mathrm{D}}} 

\DeclareMathOperator{\kernel}{kernel}

\usepackage{mathtools}

\DeclarePairedDelimiter\norm{\lVert}{\rVert}

\DeclareMathOperator{\SE}{SE}
\DeclareMathOperator{\SO}{SO}
\DeclareMathOperator{\ad}{ad}

\newtheorem{thm}{Theorem}
\newtheorem{lem} [thm]{Lemma}
\newtheorem{prop}[thm]{Proposition}

\newtheorem{defn}[thm]{Definition}
\newtheorem{assn}[thm]{Assumption}

\theoremstyle{remark}
\newtheorem{rmk}[thm]{Remark}

\title{The role of symmetry and dissipation in biolocomotion}

\author{Jaap Eldering}
\address{%
  Department of Mathematics\\
  Imperial College London\\
  London SW7 2AZ\\
  United Kingdom
}
\email{j.eldering@imperial.ac.uk}

\author{Henry O. Jacobs}
\address{%
  Department of Mathematics\\
  Imperial College London\\
  London SW7 2AZ\\
  United Kingdom
}
\email{hoj201@gmail.com}

\date{\today} 

\begin{document}

\begin{abstract}
In this paper we illustrate the potential role which relative limit cycles may
play in biolocomotion.
We do this by describing, in great detail, an elementary example of reduction of a lightly dissipative system modeling crawling-type locomotion in 3D.
The symmetry group $\SE(2)$ is the set of rigid transformations of the horizontal (ground) plane.
Given a time-periodic perturbation, the system will admit a relative limit cycle whereupon each period is related to the previous by a fixed translation and rotation along the ground.
This toy model identifies how symmetry reduction and dissipation can conspire to create robust behavior in crawling, and possibly walking, locomotion.
\end{abstract}

\maketitle

\section{Introduction}
  The notion of limit cycles is important in biolocomotion because simple periodic behavior is a defining characteristic of walking, running, swimming, flapping flight,...
  In order to construct realistic mathematical models that exhibit limit cycles, it is helpful to first identify some core mechanisms of limit cycle production.
  In particular, a biolocomotive gait, such as skipping or crawling, has three primary ingredients:
  \begin{enumerate}
  \item it is time-periodic;
  \item with each period the body translates and rotates in space;
  \item it is robust to noise and systemic variations.
  \end{enumerate}
  The combination of these ingredients is known to dynamical system theorists as an $\SE(n)$-relative limit cycle.
  Here $\SE(n)$ is the special Euclidean group for $\R^n$ (i.e.\ rotations and translations of $\R^n$).
  Such an object is a trajectory of a dynamical system with $\SE(n)$-symmetry,
  such that its image is a limit cycle under reduction by $\SE(n)$.
  In summary, an $\SE(n)$-relative limit cycle is just a limit cycle modulo rotations and translations.

  Presently much of the literature on biolocomotion concerns the search for limit cycles without addressing the role of symmetry.
  Such limit cycles are made robust through a mixture of dissipation and the dimension reduction which occurs across the transition maps in hybrid systems.
  Many of these systems occur in a regime with a mixture of friction and inertial forces.
  In contrast, the symmetry and reduction theoretic aspects of biolocomotion are very well studied in the geometric mechanics community,
  but only in regimes where friction or inertial forces dominate~\cite{Purcell1977,ShapereWilczek1989,Montgomery1990,Montgomery1993,Ostrowski1994,Koiller1996,KellyMurray2000,Kanso2005}.
  The mixed regime is mostly left unstudied by the geometric mechanics community.
  The goal of this article is to address these gaps by illustrating a simple example of an $\SE(2)$ invariant system which models crawling
  via an $\SE(2)$-relative limit cycle.
  We combine techniques from geometric mechanics, hyperbolic stability and singular perturbation theory.
  Through this analysis, one can see how these techniques can be generalized to more sophisticated and realistic models.

\subsection{Outline of the paper}
  We start with an overview of the background and motivation in section~\ref{sec:background}.
  In section~\ref{sec:connections} we review the geometric mechanics of biolocomotion in the viscous dominated and inertial regimes.
  We find that the use of connections in the middle regime is less natural.
  This lack of naturality motivates the approach of the present paper, which implements Lagrangian reduction without the explicit use of the mechanical or the Stokes connections.
  Then, in section~\ref{sec:model}, we introduce our model of an (unactuated) crawler: a mass-spring system resting on the ground, see Figure~\ref{fig:3D-crawler}.
  We regularize the no-slip and no-penetration conditions imposed by the ground by `smearing them out' over a small region around $z = 0$ to smooth viscous friction (c.f.~\cite{Brendelev81,Karapetian81}) and smooth potential energies (c.f.~\cite{Rubin57,Takens80}).
  What results from this regularization is a constant dimension, lightly dissipative Lagrangian system.
  Having described the problem as an ODE on a space of constant dimension, we apply symmetry reduction and smooth dynamical systems theory in section~\ref{sec:analysis}.
  The system is invariant under isometric transformations along the ground, and so
  we implement reduction by this group of transformations~\cite{MarsdenWeinstein1974,CeMaRa2001}.
  Under mild regularity conditions (Assumption~\ref{ass:stability-assumptions}), we can use singular perturbation theory to find a robustly stable equilibrium for the reduced model.
  Next, under small time-periodic forcing (i.e.~actuation of the crawler) this equilibrium persists as a limit cycle in the symmetry reduced model as a result of the persistence theorem~\cite{Fenichel1971,Hirsch77}.
  The limit cycle in the reduced space corresponds to a \emph{relative limit cycle} in the original phase space, (see Figure~\ref{fig:commu}). A phase reconstruction formula gives the phase shift of the lifted, relative limit cycle; this phase shift corresponds to a translation and rotation achieved upon traversing a period of the relative limit cycle.
  Both relative periodicity and stability are characteristics of biolocomotion, and so we can consider the relatively periodic orbit as a model of crawling in this sense.

  \begin{figure}[htb]
    \centering
    \includegraphics[width=6cm]{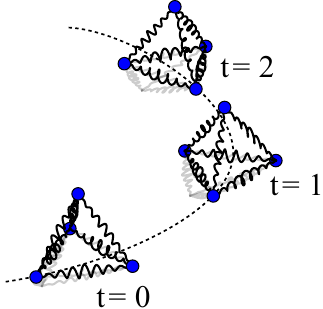}
    \caption{An illustration of a 3D crawler moving on the plane.}
    \label{fig:3D-crawler}
  \end{figure}

  \begin{figure}[ht!]
    \centering
  \begin{tikzpicture}[thick, scale=0.70,space/.style={rectangle, fill=white,draw=black,rounded corners}, text width=3cm, align=center]
	\node[space,scale=0.75] (TL) at (0,3.5) {unactuated system};
	\node[space,scale=0.75] (BL) at (0,0) {stable point in reduced space};
	\node[space,scale=0.75] (TR) at (7,3.5) {relative limit cycle \\ {\bf (e.g.~crawling)} };
	\node[space,scale=0.75] (BR) at (7,0) {stable limit cycle};
	\draw[->] (TL)--node[above,scale=0.75]{periodic force} (TR);
	\draw[->] (BL)--node[below,scale=0.75]{periodic force} (BR);
	\draw[->] (TL)--node[left,scale=0.75]{reduction} (BL);
	\draw[->] (TR)--node[right,scale=0.75]{reduction} (BR);
  \end{tikzpicture}
  \caption{This commutative diagram illustrates crawling emerging as a relative limit cycle from a stable, unactuated system.}
  \label{fig:commu}
  \end{figure}
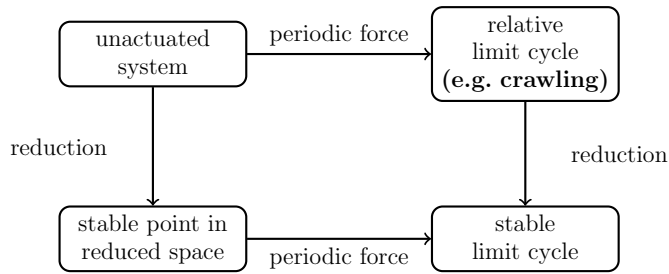

These results are summarized by the following theorem.
\begin{thm}[main theorem]\label{thm:limit-cycle}
  Let a simple crawler model be described by the Lagrange--d'Alembert equations~\eqref{eq:full-EoM}.
  Under Assumption~\ref{ass:stability-assumptions} the symmetry reduced system has a stable rest state.
  For sufficiently small time-periodic forcings, this rest state persists as a stable limit cycle, which corresponds to a relative limit cycle in the unreduced system that models crawling.
\end{thm}
We also show that the phase shift of the relative limit cycle depends on the magnitude of the perturbation to second order.

Finally, in section~\ref{sec:numerical-sim} we illustrate the theory with numerical experiments before stating our closing remarks in the conclusion.

\section{Background \& motivation} \label{sec:background}
  Biomechanics requires knowledge from a range of fields.
  This particular paper draws upon previous research in geometric
  mechanics, stability theory, as well as inspiration from
  experimental and numerical observations.

\subsection{Contact problems}
  The regularization we are going to pursue is in contrast to the hybrid systems approach, where transitions between different types of phase spaces are given by various transition maps.
  The hybrid systems approach expresses the non-constant nature of the dimension in contact problems explicitly, and has yielded a number of insights and useful models.
  For example, a hybrid systems formulation was introduced by McGeer~\cite{McGeer1990}, where the transition maps led to a dimension reduction; it was suggested that a limit cycle was approached passively.
  Since the work of~\cite{McGeer1990}, the notion of walking as a limit cycle has become more common, and more sophisticated analyses have lent further support to this idea~\cite{GaChaRuCo98,GoThuEs98}.
  The most compelling arguments are the original videos of McGeer which accompany~\cite{McGeer1990}.

\subsection{Biology and Engineering}
    On the biological side, `central pattern generators' (CPGs) have been hypothesized as fundamental neural mechanism used in biolocomotion~\cite{GrWa85}.
  These CPGs are non-localized collections of neurons which produce rhythmic activity, and respond to various inputs which modulate these rhythms.
  Therefore the link between CPGs and limit cycle biolocomotion is one which links periodic activation of the controls to periodic motion of the body.
  This link is used in the creation of simple models which can be feasibly analyzed (see for example~\cite{GoThuEs98}).
   A similar regularization of ground contact which will be presented in this paper is used in~\cite{Verdaasdonk2009}, which studies robustness and efficiency of a simple passive dynamic walking model actuated by CPGs, although the role of symmetry was not addressed there.

  The notion of a CPG is significant from the perspective of biologically inspired control theory because less demand is placed upon the control law when locomotion is achieved
  primarily through an open-loop control.
  For example, under weak assumptions, the existence of limit cycles in hybrid systems implies the existence of a reduced order model for the system as a whole~\cite{BuReSa2011}.

\subsection{Geometric mechanics}
  There is a long history of using geometric mechanics to study locomotion.
  Purcell's three link swimmer~\cite{Purcell1977}
  inspired Shapere and Wilczek to interpret locomotion in Stokes flow as
  phase shift due to the curvature of a principal connection~\cite{ShapereWilczek1989}.
  The simplicity of this perspective has
  proven useful in other dissipation dominated systems such as granular media~\cite{Hatton2013}.
  It was later found that a range of examples
  of locomotion fit within this geometric framework~\cite{Montgomery1993,Koiller1996}.
  In particular, many conservative systems could be analyzed in this way~\cite{Montgomery1990,KellyMurray2000,Kanso2005,Ostrowski1994}.

  Despite the success of the gauge theoretic picture of locomotion, the
  vast majority of examples of this perspective concern
  systems which are either conservative (i.e.~Hamiltonian or Lagrangian), or friction dominated
  (i.e.~where Newton's law, $\ddot{q} \propto F$, is replaced
  by $\dot{q} \propto F$).  There appear to be very few examples which
  invoke the gauge theoretic perspective of~\cite{ShapereWilczek1989}
  in a regime which exhibits a mixture of viscous and inertial forces.
  The paper~\cite{KellyMurray1996} by Kelly and Murray is a notable
  exception, where they discuss both mechanical and Stokes connections
  and study control of systems in the intermediate regime by viewing
  friction as a drift term to the system with mechanical connection.
  They also note that the mechanical and Stokes connections \emph{cannot}
  simply be interpolated.
  This gap between the inertial and viscous regimes is one which the current paper seeks to address.
  
  Again, the middle regime has been shown to be more than merely an interpolation
  between the two extreme regimes.
  For example, the scallop theorem states that a system in the viscosity dominated regime with only
  one degree of freedom in shape space cannot move.\footnote{This is a slight simplification which assumes that the shape space has a trivial first Homology group.
  	Nonetheless this is the ``popular'' conception of the scallop theorem, for better or worse.}
  It was shown that the scallop theorem is violated if one modifies the friction and allows for inertial forces to play a role~\cite{WagnerLauga2013}.
  The geometry of this system was not explored, but~\cite{WagnerLauga2013}
  provided an insightful counter example to the scallop theorem in the regime where inertial and viscous forces both play sizable roles.
  Indeed, in this paper we want to argue that the gauge theoretic
  picture of~\cite{ShapereWilczek1989,KellyMurray1996,Montgomery1990} using connections does not persist, or at least not in a
  clear way. Instead, we propose a more general geometric framework,
  describing the symmetry on the vector bundle that follows
  from quotienting phase space by the symmetry, without invoking the
  Lagrange--Poincar\'e decomposition which results from choosing a
  connection. The phase shift can be recovered from a reconstruction
  formula that is implicitly specified by a dynamical perturbation
  argument. In the next section we give a short overview of the gauge
  connection picture and how our setting addresses a gap within this picture.

 As a final point, the role of symmetry is well acknowledged within the geometric mechanics literature but 
 this is not to say that it is absent from the biomechanics literature, for example the importance of discrete symmetries of solutions is widely acknowledged~\cite{Raibert1986a,Raibert1986b}.
 However, tools such as momentum maps and connections are typically not used.
 This is possibly due to the fact that geometric mechanics has only addressed the extreme regimes, 
 while many popular biomechanical models fall within the middle regime.
 A notable exception is~\cite{GreggRighetti2013}, where the Noetherian
 momentum associated to an $S^1$~symmetry
 was used to create turning trajectories based upon the work of~\cite{AmesThesis,AmesSastry2006}.
 Here we will be exploring a different, but related, application of symmetry reduction
 where Noether's theorem is never invoked (nor does it apply).

\section{The gap between the mechanical and Stokes connections}
\label{sec:connections}

In this section we will explore the traditional use of connections
in understanding locomotion.
We will find that the use of connections is
unmotivated when there is a mixture of inertial and viscous forces.
This section is
aimed at an audience which is familiar with these more established techniques.
As the goal of the section is to illustrate why we should
\emph{not} use these tools, we recommend that the reader should skip this section
if she is unfamiliar with the use of connections in locomotion, at least upon a first reading.

The typical setup of the configuration space in biolocomotion is that
of a principal $G$-bundle, where $G$ is the (spatial) symmetry group
of the system and generates the directions in which locomotion can
take place. That is, we have a configuration space $Q$ and a left
$G$-action on $Q$, such that the quotient projection
$\pi\colon Q \to S := G \backslash Q$ is a left $G$-principal bundle.
The base $S$ is conventionally called the `shape space', as it
describes the state of the system modulo its position.

Assuming that the system is symmetric under $G$, we can consider the
reduced dynamics on $G \backslash \T Q$.
This bundle is locally isomorphic to $\T S \times \mathfrak{g}$
and is naturally a vector
bundle over $S$; in the fibers, all velocities are retained, since the
dynamics may still depend on these velocities even though it does not
depend on the underlying points in the orbits of $G$.

Given a connection, this can be decomposed in a vector bundle sum
\begin{equation}\label{eq:LP-splitting}
  G \backslash \T Q \cong \T S \oplus \tilde{\mathfrak{g}}
\end{equation}
where $\tilde{\mathfrak{g}}$ is the adjoint bundle and a natural
vertical distribution in $G \backslash \T Q$, while the connection is
used to identify $\T S$ with a horizontal distribution.

Now there are two natural and useful choices of connection in the
limit cases (see e.g.~\cite{KellyMurray1996}): the mechanical
connection, for when the dynamics is Hamiltonian, and the Stokes
connection in case of a friction dominated limit, i.e.\ high Reynolds
number in swimming-like locomotion or high Froude(-like) number in
terrestrial, finite-dimensional locomotion models.
We shall see that in the middle regime there is generally no natural choice of connection in order to understand locomotion.
This lack of a natural choice will motivate us to avoid the use of a connection later in the paper.

The goal of this section is to illustrate this inability to address the middle regime.
We shall begin by discussing the general setup of a dissipative
mechanical system on a manifold before describing the role of connections in understanding locomotion.

\subsection{Lagrangian mechanics with dissipation}

Let us briefly return to a setup without symmetry assumption.
A mechanical system with (viscous) friction can be represented
using a Lagrangian to model the conservative part and a Rayleigh
function to model friction. Let
\begin{equation*}\label{eq:Lagrangian}
  L(q,\dot{q}) = \frac{m}{2} k_q(\dot{q},\dot{q}) - V(q)
\end{equation*}
denote the Lagrangian with kinetic energy metric $k$ which turns $(Q,k)$
into a Riemannian manifold, and with potential $V$. Let
\begin{equation*}\label{eq:Rayleigh-function}
  R(q,\dot{q}) = \frac{c}{2} \nu_q(\dot{q},\dot{q})
\end{equation*}
denote a Rayleigh dissipation function that defines a friction force
given by minus its fiber derivative, see~\cite[Def.~3.5.2]{FOM}. Note
that $\nu_q$ is assumed to be a positive definite quadratic form on
$\T_q Q$ that depends smoothly on $q$, hence $\nu$ is a metric on $Q$,
like $k$. The parameters $m$ and $c$ will allow us to consider the
Hamiltonian and dissipation dominated limits.

Now the Lagrange equations of motion are given by
\begin{equation}\label{eq:EoM}
  m\,k^\flat \cdot \nabla^k_{\dot{q}} \dot{q}
  = -\d V(q) - c\,\nu_q^\flat \cdot \dot{q},
\end{equation}
to which an extra force $F \in \T^* Q$ can be added on the right-hand
side. When we take the limit of $c \to 0$, we straightforwardly
converge to a conservative system. When $m \to 0$, a more careful
singular perturbation analysis shows (using the assumption that $\nu$
is positive definite, see Appendix~\ref{app:friction-dynamics}) that
\begin{equation}\label{eq:invar-1st-order-dynamics}
  M = \{ c\,\nu_q^\flat \cdot \dot{q} = -\d V(q) \} \subset \T Q
\end{equation}
is a well-defined, attractive invariant manifold for the limit
dynamics. We can view $M$ as a submanifold of $\T Q$, but since $M$ is
the graph of a section of $\T Q$, we can also view it as a vector
field that generates first order dynamics on $Q$. This is the
Stokesian limit
\begin{equation}\label{eq:EoM-Stokes}
  \dot{q} = -\frac{1}{c} \nu_q^\sharp \cdot \d V(q).
\end{equation}

\subsection{Mechanical connections}

If $k$ is a kinetic energy metric on $Q$ that is invariant under $G$,
then it defines a connection on $\pi\colon Q \to S$ by defining the
horizontal space $\textrm{Hor}^k(\T Q)$ complementary to
$\textrm{Ver}(\T Q) := \mathfrak{g}\cdot Q \subset \T Q$ as its
perpendicular under $k$. This can be viewed as a sub vector bundle
$\textrm{Hor}^k(\T Q) \subset \T Q$ and descends through the quotient
by $G$ to a sub vector bundle of $G \backslash \T Q$. Note that this
connection is \emph{not} the Levi-Civita connection defined by $k$,
since the Levi-Civita connection induces a splitting of $\T \T Q$, at
one higher level.

Now if the complete system is invariant under the (lifted) action of
$G$, then by symmetry, $\textrm{Hor}^k(\T Q)$ is an invariant
submanifold for the dynamics, and it is exactly the level set of zero
momentum under the (Lagrangian) momentum map induced by the action of
$G$. If we use this connection for the
identification~\eqref{eq:LP-splitting}, then it implies that the
$\tilde{\mathfrak{g}}$ component is constantly zero, hence we can
reduce to dynamics on $\T S$, and after solving that, lift solution
curves to $G \backslash \T Q$ and even $\T Q$, using the mechanical
connection and integration of the vertical component, respectively.

In this case the equations of motion~\eqref{eq:EoM} reduce to
\begin{equation}\label{eq:EoM-Hamiltonian}
  m\,k^\flat \cdot \nabla^k_{\dot{q}} \dot{q} = -\d V(q)
\end{equation}
and exhibit $\textrm{Hor}^k(\T Q)$ as invariant submanifold.

\subsection{Stokes connections}

The Stokes connection is defined in the same way as a mechanical
connection, but now using the metric $\nu$ on $Q$. In this case we
assume that friction forces dominate the inertial forces and, hence,
the dynamics is only first order, and given by~\eqref{eq:EoM-Stokes}.
Next, it is typically assumed, e.g.\ in Stokes flow swimming, that
there is an external force $F(t)$ exerted by the swimmer, say, which
physically implies that $F(t) \in \textrm{Ver}(\T Q)^0$, the
annihilator of $\textrm{Ver}(\T Q)$. Since $V$ was assumed
$G$ invariant it follows that $\d V(q) \in \textrm{Ver}(\T Q)^0$ too,
and using the fact that
$\nu^\sharp\big(\textrm{Ver}(\T Q)^0\big) = \textrm{Hor}^\nu(\T Q)$ by
definition, we obtain the control system
\begin{equation}\label{eq:control}
  \dot{q} = f(q) + u(t) \in \textrm{Hor}^\nu(\T Q)
\end{equation}
where $f(q) = -\frac{1}{c} \nu_q^\sharp \cdot \d V(q)$ and
$u(t) = \nu_q^\sharp \cdot F(t)$. In particular, given a curve
$s(t) \in S$, we can lift it to a curve $v(t) \in G \backslash \T Q$
using the Stokes connection and~\eqref{eq:control}. This corresponds to
the unique motion in $Q$ such that no work is done in the directions
of the symmetry $G$. See also~\cite{KellyMurray1996} and more details in
Appendix~\ref{app:friction-dynamics}.

\subsection{The middle regime}

Let us now return to study dynamics on $\pi\colon Q \to S$ in the
middle regime where both inertial and frictional forces are present,
i.e.\ neither $m$ nor $c$ is negligibly small. We can rewrite the
equations of motion~\eqref{eq:EoM} as
\begin{equation*}
  m\, \nabla^k_{\dot{q}} \dot{q}
  = - k^\sharp \cdot \d V(q)
    - c\, k^\sharp \cdot \nu^\flat \cdot \dot{q}.
\end{equation*}
Since $\d V(q) \in \textrm{Ver}(\T Q)^0$ it follows that the first
right-hand side term lives in $\textrm{Hor}^k(\T Q)$, hence that part
of the dynamics preserves the splitting
$\textrm{Hor}^k(\T Q) \oplus \textrm{Ver}(\T Q)$. The mapping
\begin{equation*}
  k^\sharp \cdot \nu^\flat\colon \T Q \to \T Q,
\end{equation*}
however, will generally not preserve this splitting, so the mechanical
connection does not yield a reduction here. If $m > 0$ is sufficiently
small, we can actually still reduce to a first order system. This can
be considered the `perturbed Stokes regime' where the first order
ODE~\eqref{eq:EoM-Stokes} does not accurately hold anymore, but
approximations can be found using singular perturbation theory. The
corrections, however, cannot be interpreted as a connection anymore,
see Appendix~\ref{app:friction-dynamics} for the details.

The lack of a natural connection in the middle regime suggests that we should implement
reduction by symmetry without the use of a connection.
In the language of geometric mechanics, this means we shall derive equations of motion on the Atiyah algebroid $\T Q / G$, as in~\cite{Weinstein1996}, as opposed to a Lagrange--Poincar\'e bundle $( \T(Q/G) \oplus \tilde{\mathfrak{g}} , \d A)$, as in~\cite{CeMaRa2001}.

\section{The model}\label{sec:model}
  The model can be broken into two distinct components: the crawler and the environment.
  The crawler consists of four masses connected by springs while the environment consists of the ground and a gravitational field.
  We will discuss the model of the crawler in empty space before we elaborate on how to include interactions with the environment.

\subsection{A model of a crawler (in a vacuum)}

  The crawler consists of four point particles of unit mass all connected by springs of stiffness $\kappa_\s$ with light viscous damping $c_\s$, see Figure~\ref{fig:3D-crawler}. We describe the crawler as a Lagrangian mechanical system with additional forces to model the spring damping.
  The point particles move through space with positions ${\bf x}_i = (x_i,y_i,z_i) \in \R^3$ and velocities $\dot{\bf x}_i = (\dot{x}_i,\dot{y}_i,\dot{z}_i) \in \R^3$ for $i = 1,2,3,4$.
  For reasons to be clarified shortly, we will exclude configurations where any of the particles overlap,
  and the configurations where all of the $(x,y)$ coordinates overlap.
  Thus the configuration space is a (dense) open set $Q \subset \R^{12}$.
  We will use ``$q$'' to denote a generic point of $Q$ and $(q^1,\dots,q^{12})$ to denote generalized coordinates of $Q$.

  The kinetic energy is given by $T = \frac{1}{2} \sum_{i=1}^4 \|\dot{\bf x}_i\|^2$ with the usual Euclidean metric.
  This endows $Q$ with a flat Riemannian structure, and we will denote the Riemannian metric by $k$ or $k_{ij}(q)$ in generalized coordinates.
  The potential energy from the springs, $U_\s$, is more easily expressed in other coordinates: the spring lengths,
  i.e.\ the pair-wise distances between the points ${\bf x}_1,\dots,{\bf x}_4$.
  We therefore introduce six (local) coordinate functions
  \begin{equation*}\label{eq:length-variables}
    \ell_{ij} = \norm{ {\bf x}_i - {\bf x}_j} = \sqrt{ (x_i -x_j)^2 +(y_i-y_j)^2+ (z_i - z_j)^2 }
  \end{equation*}
  for $i < j$ and $i,j = 1,\dots,4$.
  The potential energy of the springs is now simply given by
  \begin{equation}\label{eq:spring-potential}
    U_\s = \frac{\kappa_\s}{2} \sum_{i < j} \left( \ell_{ij} - \bar{\ell}_{ij} \right)^2
  \end{equation}
  where $\bar{\ell}_{ij}$ are constants which denote the rest length of spring between ${\bf x}_i$ and ${\bf x}_j$.

  We define the viscous force of each spring by the one-form
  \begin{equation}\label{eq:Fs-friction-ell}
    F_{ij} = -c_\s\dot{\ell}_{ij} \d\ell_{ij}.
  \end{equation}
  In terms of the usual ${\bf x}_i = (x_i,y_i,z_i)$ coordinates these six forces can be written as a sum of twelve force vectors $\mathbf{F}_{ij}$ describing the force exerted on particle $i$ by the viscous friction of the spring connecting it to particle $j$.
  We have
  \begin{equation}\label{eq:Fs-friction-xy}
    \mathbf{F}_{ij} = - c_\s \frac{ \lb \dot{\bf x}_i - \dot{\bf x}_j , {\bf x}_i - {\bf x}_j \rb }{ \ell_{ij}^2 } ({\bf x}_i - {\bf x}_j)
                    = - c_\s \frac{\d\,\norm{ {\bf x}_i - {\bf x}_j}}{\d t} \hat{\bf{n}}_{ij}
  \end{equation}
  where $\hat{\mathbf{n}}_{ij}$ is the unit vector pointing from mass $j$ to mass $i$.
 The expression~\eqref{eq:Fs-friction-xy} constitutes the components of~\eqref{eq:Fs-friction-ell} with respect to the standard basis one-forms $(\d x_i,\d y_i,\d z_i)$.
 More precisely, if we denote the components of $\mathbf{F}_{ij}$ by $\mathbf{F}_{ij}^x$ $\mathbf{F}_{ij}^y$ and $\mathbf{F}_{ij}^z$, then the sum $\mathbf{F}_{ij}^x \d x_i + \mathbf{F}_{ij}^y \d y_i + \mathbf{F}_{ij}^z \d z_i$ is the one-form acting upon mass\footnote{As the one-form $\d x_i$ is independent of $\d x_j$ when $i \neq j$ we see that $F_{ij} \neq - F_{ji}$ as one-forms.} $i$, and we have $F_{ij} = \mathbf{F}_{ij}^x \d x_i + \mathbf{F}_{ji}^x \d x_j + \dots$.
 Thus, expression~\eqref{eq:Fs-friction-ell} conveniently captures the string damping force applied to the particles at both its endpoints.
  We see that the viscous friction forces oppose length change of the springs, exactly as expected.
  In any case, we can define\footnote{%
    Equations~\eqref{eq:spring-potential} and~\eqref{eq:Fs-friction-xy} (with the expression for $\ell_{ij}$ substituted) show that the system is ill-defined when ${\bf x}_i = {\bf x}_j$ for some $i \neq j$.
  This is a set of positive codimension which we shall stay away from in our analysis.%
} the force $F_\s = \sum_{i<j} F_{ij} $.

  Later in the paper we will make the rest lengths $\bar{\ell}_{ij}$ time dependent as a means to indirectly control the actual lengths of the springs.
Upon performing the substitution by functions $\bar{\ell}_{ij}(t)$, one should be careful about what kind of system is modeled by the resulting equations of motion.
In our case, one could imagine that the viscous damping is realized through the addition of dashpots being placed in parallel to the springs.

  \subsection{A regularized model of the ground}
  The ground is described by the plane $\{ z=0 \}$ in $\R^3$.
  Ideally, the ground is impenetrable and imposes a no-slip condition, mathematically represented by the constraints
  \begin{align}
    &z_i \geq 0, \label{eq:no_pen} \\
    &(\dot{x}_i,\dot{y}_i) = (0,0) \text{ if } z_i = 0 \label{eq:no_slip}
  \end{align}
    for $i=1,2,3,4$, where equation~\eqref{eq:no_pen} is the \emph{no-penetration condition} and equation~\eqref{eq:no_slip} is the \emph{no-slip condition}.
   Both conditions present challenges of a singular nature because they abruptly `turn on' at $z=0$ and are inactive otherwise.
  It is precisely this `on/off' character which we will regularize.
  To do this we will repeatedly make use of the differentiable\footnote{%
  The function $\chi$ is of class $C^1$ only.
  However, this can be dealt with by applying a smoothing mollifier concentrated around $0$.
  The width of the mollifier can be made arbitrary small, such that it does not overlap the fixed point to be found in Proposition~\ref{prop:potential-minimum}; this prevents any possible circular dependencies in size estimates later on.
  Thus, without loss of generality we may assume that the system is smooth by viewing $\chi( \cdot )$ as a proxy for a smooth function with the same behavior away from $0$.%
  } function
    \[
    	\chi(z) = \begin{cases}
		\frac{1}{2} z^2 & \text{ if } z < 0, \\
		0               & \text{ else}
		\end{cases}
    \]
  to construct forces and potentials.

  We approximate the no-penetration condition by considering a potential energy that grows rapidly for each $z_i < 0$ and is zero when $z_i \ge 0$ for $i=1,2,3,4$.
  Therefore, we define the potential energy $U_{\np}\colon Q \to \R$ by
  \begin{equation}\label{eq:floor_pot}
    U_{\np}(q) = \kappa_{\np} \sum_{i=1}^{4} \chi(z_i).
  \end{equation}
  This penalizes particles for falling through the floor and the penetration depth for a particle at rest can be controlled with $\kappa_{\np}$.
  When $\kappa_{\np}$ approaches infinity, the penetration depth goes to zero and our model approaches an exact model of a perfectly impenetrable ground.
  This can be viewed as modeling a one-sided holonomic constraint in the spirit of~\cite{Rubin57,Takens80}.
  A more advanced version of such an approach is used in~\cite{VoHaTaGr2011} to model contact problems with accurate simulations without being slaved to using infinitesimal time-step sizes.

  The no-slip condition is similar to the no-penetration condition in that it is only active at $\{ z=0 \}$.
  However, unlike the no-penetration condition, the no-slip condition is not derivable from a potential energy but instead can be viewed as a limit of viscous friction~\cite{Brendelev81,Karapetian81}.
 In particular, consider the viscous force given by
  \[
    F_{\ns}(q,\dot{q})=  - c_{\ns} \sum_{i=1}^{4}{ \chi' \left( z_i \right) ( \dot{x}_i \d x_i + \dot{y}_i \d y_i ) }.
  \]
  The force $F_{\ns}$ dampens the horizontal motion of particles in a region around $\{ z=0 \}$.
  Moreover, we can see that $F_{\ns}$ is proportional to $\d U_{\np}$, the normal force exerted by the ground.
  This is consistent with standard (first-order) assumptions about the nature of slip-friction.
  As before, the coefficient $c_{\ns}$ controls the amplitude of this force and when $c_{\ns}$ goes to infinity we arrive at a no-slip condition.

  Similarly, we dampen bouncing at the impact of a particle with the ground by
  including the viscous friction force
  \[
    F_{\db}(q,\dot{q}) =  - c_{\db} \sum_{i=1}^{4} \chi\left( z_i  \right) \dot{z}_i \d z_i.
  \]

  Finally, we incorporate gravity via the potential energy
\[
  U_{\g}(q) = \sum_{i=1}^{4} z_i
\]
 which imposes the gravitational force $-\d U_{\g}(q) = -\sum_{i=1}^{4} \d z_i$.

  \subsection{The full model}
  Now that we have established the Lagrangian of the crawler, as well as the environmental forces imposed on it, we can finally provide the equations of motion.
  These equations of motion are obtained by adding the viscous forces, $F_{\ns}$ and  $F_{\db}$, and the potential forces, $-\d U_{\np}$ and $-\d U_{\g}$, to the equations for the crawler in a vacuum.
  Adding these up into the total potential energy $U = U_\s + U_{\np} + U_{\g}$ and the total force $F = F_\s + F_{\ns} + F_{\db}$,
  the equations of motion are the Lagrange--d'Alembert equations,
  \begin{equation}
  	\frac{d}{dt} \left( \pder{L}{\dot{q}}(q,\dot{q}) \right) - \pder{L}{q}(q,\dot{q}) = F(q,\dot{q}), \label{eq:LDA}
  \end{equation}
  where $L = \frac{1}{2} k_{ij} (q) \dot{q}^i \dot{q}^j - U(q)$.
  This equation implicitly determines $\ddot{q}$ given  $q$ and $\dot{q}$.
  We can make this expression more explicit by writing it in the form
  \begin{equation}\label{eq:full-EoM}
    \frac{\d^2 q^i}{\d t^2} + \Gamma^{i}_{jk}(q) \frac{\d q^j}{\d t} \frac{\d q^k}{\d t}
    = k^{ij}(q) \big( F_j(q,\dot{q}) - \partial_j U(q) \big),
  \end{equation}
  where $k^{ij}(q)$ denotes the cometric and $\Gamma^i_{jk}(q)$ the Christoffel symbols associated to the metric $k$.
  Note that $F(q,\dot{q})$ is linear in the velocity and can be written as $F(q,\dot{q}) = -\nu(q) \cdot \dot{q}$
  for a positive semi-definite quadratic form $\nu(q)$ given by
  \begin{equation}\label{eq:viscosity_matrix}
    \nu_{ij}(q) \, \dot{q}^i \, \dot{q}^j
    := \left(\sum_{i=1}^4 c_\db \, \chi(z_i) \dot{z}_i^2
                         +c_\ns \, \chi'(z_i) (\dot{x}_i^2 + \dot{y}_i^2)\right)
      +\left(\sum_{i < j} c_\s \, \dot{\ell}_{ij}^2 \right).
  \end{equation}
  In fact we will find that $\nu$ is positive definite (see Proposition~\ref{prop:nondegenerate}, page~\pageref{prop:nondegenerate}).

\section{Analysis} \label{sec:analysis}

  In this section we prove the existence of a robustly stable equilibrium in a symmetry reduced phase space.
  To begin, we review the general process of reduction by symmetry before handling the specific case at hand.
  We reduce our system by an $\SE(2)$ symmetry to obtain a reduced vector field on the reduced phase space $\SE(2) \backslash TQ$.
  Subsequently, we prove the existence of a robustly stable equilibrium which can then be periodically perturbed to obtain a limit cycle. We reconstruct from it a relative limit cycle in the unreduced system.
  Finally, we provide some illustrative numerical results to support our claim that the reconstructed relative limit cycle typically has a non-trivial phase shift.

\subsection{Reduction by symmetry in general}
The notion of reduction by symmetry in dynamical systems is conceptually very simple.
If a system is invariant under a group of transformations, then it is, in some sense, more simple
than a system which is not invariant.

Simply put if $X:M \to TM$ is a vector field on $M$ and $G$ is a Lie group which acts on $M$,
then we say that $X$ is invariant under $G$ if $X(g \cdot x) = g \cdot X(x)$ for all $x \in M$ and $g \in G$.
Here $g$ acts on $TM$ by the tangent lift of the action on $M$.
For example if $M = \R^2$ and $G = \SO(2)$ acts on $\R^2$ by rotation about the origin,
then a vector field is $G$ invariant if it is of the form $X(r,\theta) = f_\theta(r) \pder{}{\theta} + f_r(r) \pder{}{r}$.

The quotient\footnote{%
  We implicitly always use left actions, and therefore write the group that is quotiented out on the left.%
} space $G \backslash M$ is the space of $G$-orbits.  In our example $G \backslash M$ is the space of
circles centered at the origin, which can be identified with $\mathbb{R}^+$.
In the case of our system $M = TQ$, $G = \SE(2)$ and $G \backslash M = \SE(2) \backslash TQ$ is a space which stores
the shape of the mass-spring system and its velocity, but not its position on the ground.
If $G$ acts on $M$ freely and properly, then $G \backslash M$ is a smooth manifold
and there is a smooth surjection $\Pi : M \to G \backslash M$ which sends each point $ x \in M$ to its $G$-orbit $G \cdot x \in G \backslash M$.
Moreover, if the vector field $X:M \to TM$ is $G$ invariant, then there exists a unique vector field $Y : G \backslash M \to \T(G \backslash M)$
such that $\T\Pi \cdot X = Y \circ \Pi$.

In our example, if $X = f_\theta(r) \pder{}{\theta} + f_r(r) \pder{}{r}$, then $G \backslash M = \R^+$
is coordinatized by the radius $r$ and $Y(r) = f_r(r) \pder{}{r}$.
We see that $Y$ describes dynamics on a smaller space ($G \backslash M$), yet still captures all of the richness of $X$.
Determining $Y$ from a vector field $X$ is known as \emph{reduction by symmetry}.
In the next section we will perform reduction by symmetry with respect to the group of rotations and translations of the plane.

Finally, if $x \in M$ is an equilibrium of $X: M \to TM$ and $Y: G \backslash M \to \T(G \backslash M)$ is obtained via reduction by symmetry,
then $y = \Pi(x)$ is an equilibrium of $Y$.
Moreover, the linearization of $Y$ about $y$ is related to the linearization of $X$ about $x$.

\begin{prop}\label{prop:linearization}
  Assume $G$ acts freely and properly on $M$, and let $\Pi : M \to G \backslash M$ denote the quotient projection.
  Let $x \in M$ be a fixed point of $X \in \mathfrak{X}(M)$.
  If $X$ is $G$ invariant, then $y = \Pi(x)$ is a fixed point of the reduced vector field $Y \in \mathfrak{X}(G \backslash M)$ and the linearization
  of $Y$ about $y$ is given by $\T_y Y = \T_0(\T_x \Pi) \cdot T_x X \cdot (\T_x \Pi)^{-1}_{\rm right}$,
  where $(\T_x \Pi)^{-1}_{\rm right}$ is an arbitrary right inverse to $\T_x \Pi$.
\end{prop}

\begin{lem}\label{lem:kernel-subset}
  Assume the setup of Proposition~\ref{prop:linearization}.
  Then the kernel of $\T_x \Pi$ is a subset of the kernel of $\T_x X\colon \T_x M \to \T_0(\T_x M)$.
\end{lem}
\begin{proof}
  Let $\Phi_t^X:M \to M$ denote the flow of the vector field $X$.
  As a consequence of~\cite[Prop.~4.2.4]{MTA} we know that $\Phi_t^X$ is $G$-invariant when $X$ is $G$-invariant.
  If $\delta x \in \T_x M$ is in the kernel of $\T_x\Pi$ then it must be of the form $\delta x = \left. \frac{\d}{\d\varepsilon} \right|_{\varepsilon = 0} g_\varepsilon \cdot x$ for some curve $g_\varepsilon \in G$ which originates at $g_0 = \textrm{id}$.
  We find
  \begin{align*}
    \T_{x} \Phi_t^{X} ( \delta x)
    &:= \frac{\d}{\d \varepsilon } \Big|_{\varepsilon = 0} \Phi_t^{X}( g_\varepsilon \cdot x )
      = \frac{\d}{\d \varepsilon } \Big|_{\varepsilon = 0} g_\varepsilon \cdot \Phi_t^{X}( x )
      = \frac{\d}{\d \varepsilon } \Big|_{\varepsilon = 0} g_\varepsilon \cdot x = \delta x.
  \end{align*}
  Therefore, $\T_x \Phi^X_t$ is the identity on the subspace of $\T_x M$ tangent to a $G$-orbit.
  Taking the time derivative we find that $\T_x X$ must evaluate to $0$ on the subspace of $\T_x M$ tangent to a $G$-orbit.
\end{proof}

\begin{proof}[Proof of Proposition~\ref{prop:linearization}]
  By the commutative relation between $X$ and $Y$ above we observe that $y = \Pi(x) \in G \backslash M$ is a fixed point of  $Y$.
  As $\T_x\Pi$ is surjective, we may define the formal inverse $(\T_x \Pi)^{-1}\colon \T_{\Pi(x) }(G \backslash M) \to \frac{\T_x M}{ \ker(\T_x \Pi) }$.
  By Lemma~\ref{lem:kernel-subset}, $\ker(\T_x \Pi) \subset \ker(\T_x X)$, so that $\T_0(\T_x \Pi) \cdot \T_x X \cdot (\T_x \Pi)^{-1}$ is a well-defined map from $\T_{\Pi(x)} G \backslash M \to \T_0(\T_{\Pi(x)} G \backslash M)$.
  In other words $\T_{\Pi(x)} Y = \T_0(\T_x \Pi) \cdot \T_x X \cdot (\T_x \Pi)^{-1}$.
  We may now replace the formal inverse, $\T_x \Pi^{-1}$, with an arbitrary right inverse, $(\T_x \Pi^{-1})_{\rm right}$ to conclude the proof.
\end{proof}

\subsection{Reduction by \texorpdfstring{$\SE(2)$}{SE(2)}}
\label{sec:reduction}
  The group $\SE(2)$ consists of all isometries of the plane, i.e.\ rotations and translations of $\mathbb{R}^2$.
  Elements of $\SE(2)$ are given by an angle, $\Theta \in S^1$, and a translation vector $\Delta = (\Delta_x,\Delta_y) \in \mathbb{R}^2$.
  We consider the action of $\SE(2)$ on $Q$ that translates and rotates the $(x,y)$ coordinates of each of the masses.
  That is, we define the action
  \begin{align*}
    (\Theta, \Delta ) \cdot ({\bf x}_1,\dots,{\bf x}_4) = \big( (\Theta,\Delta) \cdot {\bf x}_1, \dots, (\theta,\Delta) \cdot {\bf x}_4 \big)
  \end{align*}
  where
  \begin{align*}
  	(\theta,\Delta) \cdot {\bf x}_i :=
		\begin{pmatrix}
			\cos(\Theta) x_i - \sin(\Theta) y_i + \Delta_x\\
			\sin(\Theta) x_i + \cos(\Theta) y_i + \Delta_y\\
			z_i
		\end{pmatrix}.
  \end{align*}
  
  Note that $\SE(2)$ does \emph{not} act freely on all of $\R^{12}$: if all masses ${\bf x}_1,\dots,{\bf x}_4$
  have the same $(x,y)$ coordinates, then $q$ is fixed by the isometries which rotate the plane about $(x,y)$.\footnote{
  	This is called \emph{the stabilizer subgroup} of the point $(x,y)$, denoted $\SE(2)_{(x,y)}$.
  }
  We thus take the open subset $Q \subset \R^{12}$ with these configurations excluded as our configuration manifold.
  Specifically, if we let $p_{xy}:\mathbb{R}^3 \to \mathbb{R}^2$ be the projection onto the $xy$-plane, then
  \begin{align*}
    Q := \left\{ ({\bf x}_1,\dots,{\bf x}_4) \in (\R^3)^4 \,\middle|
	\begin{array}{l}
		{\bf x}_i \neq {\bf x}_j \text{ for all $i \neq j$} \\
		p_{xy}({\bf x}_i) \neq p_{xy}({\bf x}_j) \text{ for some $i \neq j$}
	\end{array} \right\}
  \end{align*}
  and we find:
  \begin{prop}
    The action of $\SE(2)$ on $Q$ is free and proper.
  \end{prop}

  \begin{proof}
    The action is free if $(\Theta, \Delta)\cdot q = q$ implies that $(\Theta,\Delta)  = (0 , \vec{\bf 0})$.
    Since the collection of $(x,y)$ coordinates of the points ${\bf x}_i$
    are prohibited from completely overlapping, it must be the case that $(\Theta , \Delta)$
    fixes a non-degenerate line segment in the plane.
    The only such isometry which satisfies this constraint is the identity, $(0,\vec{\bf 0})$.
    
    To prove that the action is proper, we have to show that
    \[
      A\colon\SE(2) \times Q \to Q \times Q\colon
      (g,q) \mapsto (g \cdot q,q)
    \]
    is a proper continuous map, see~\cite[p.~53]{DuKo2000}.
    We shall do so by proving that $A$ has a continuous inverse, defined on its image.
    Let $(q',q) \in \operatorname{Im}(A)$ and without loss of generality assume that $p_{xy}({\bf x}_1) \neq p_{xy}({\bf x}_2)$.

    Then $A^{-1}(q',q) = \big((\Theta,\Delta),q\big)$ where
    $\Theta = \angle\big(p_{xy}({\bf x}'_2 - {\bf x}'_1),p_{xy}({\bf x}_2 - {\bf x}_1)\big)$
    and $\Delta = p_{xy}(R(-\Theta) \cdot {\bf x}'_1 - {\bf x}_1)$.
    Firstly, the angle $\Theta$ depends continuously on the arguments
    vectors, since these have non-zero lengths. Secondly, the
    translation $\Delta$ depends continuously on $\Theta$ and the
    other arguments, where $R(\alpha)$ denotes the matrix of rotation
    over an angle $\alpha$.
  \end{proof}
  
  As this action is free and proper we can assert that the quotient space, $\SE(2) \backslash Q$, is a manifold, and $\pi:Q \to \SE(2) \backslash Q$ is a principal bundle.
    In order to understand the principal bundle structure of $Q$
    it is useful to find a coordinate system in which the map $\pi$ is a Cartesian projection.
    Let us consider the (local) coordinates $(\ell , Z ,\theta , x , y )$ where
    \begin{equation*}\label{eq:local-coords}
      \begin{alignedat}{2}
        \ell   &= (\ell_{12},\dots,\ell_{34}) \in (\R^+)^6,\quad &
        Z      &= (z_1,z_2,z_3), \\
        \theta &= \angle\big( (x_2-x_1,y_2-y_1) , (1,0) \big),\quad &
        (x,y)  &= \frac{1}{4} \sum_{i=1}^4 (x_i,y_i).
      \end{alignedat}
   \end{equation*}
   In words, $(x,y)$ is the average of the mass positions in the plane
   and $\theta$ is the angle between the line segment from $(x_1,y_1)$ to $(x_2,y_2)$
   and the $x$-axis.
   In these coordinates the action of $(\Theta , \Delta ) \in \SE(2)$ is given by
   \begin{align*}
     (\ell,Z,\theta,x,y) \mapsto (\ell,Z,\theta + \Theta,
     	\begin{pmatrix}
		\cos(\Theta) x - \sin(\Theta) y + \Delta_x \\
		\sin(\Theta) x + \cos(\Theta) y + \Delta_y		
	\end{pmatrix}).
   \end{align*}
   These coordinates locally trivialize $Q \cong S \times \SE(2)$ as a principal $\SE(2)$ bundle
   in that the quotient projection $\pi$ simply drops the last three coordinates $(\theta,x,y)$, and the space $S = \SE(2) \backslash Q$ is
   a nine-dimensional space with coordinates $s = (\ell ,Z)$.

  The action on $Q$ naturally lifts to a free and proper action on the tangent bundle, $TQ$, given by
  \begin{align*}
  	(\Theta , \Delta) \cdot ( ({\bf x}_1,{\bf \dot{x}}_1) , \dots , ({\bf x}_4,{\bf \dot{x}}_4) )
    := ( (\Theta , \Delta) \cdot ({\bf x}_1,{\bf \dot{x}}_1) , \dots, (\Theta , \Delta) \cdot ({\bf x}_4,{\bf \dot{x}}_4) )
  \end{align*}
  where
  \begin{align*}
  	(\Theta , \Delta) \cdot ({\bf x}_i , {\bf \dot{x}}_i) := \Bigg(
    \begin{pmatrix}
			\cos(\Theta) x_i - \sin(\Theta) y_i + \Delta_x \\
			\sin(\Theta) x_i + \cos(\Theta) y_i + \Delta_y \\
			z_i
		\end{pmatrix}
		,
		\begin{pmatrix}
			\cos(\Theta) \dot{x}_i - \sin(\Theta) \dot{y}_i \\
			\sin(\Theta) \dot{x}_i + \cos(\Theta) \dot{y}_i \\
			\dot{z}_i
		\end{pmatrix}
		\Bigg).
  \end{align*}
  
  As before, we find that $P:= \SE(2) \backslash TQ$ is a smooth manifold
  and we obtain a (left) $\SE(2)$ principal bundle $\Pi : TQ \to P$.\footnote{
    The reader should keep in mind that $P \neq \T(\SE(2) \backslash Q)$.
  }
  Also as before, in order to understand the principal bundle projection, $\Pi$,
  it is useful to use a coordinate system where $\Pi$ is trivial.
  Consider the coordinate system $(\ell,Z,\theta,x,y, \dot{\ell},\dot{Z},\dot{\theta},a,b)$
  where $\dot{Z}$, $\dot{\ell}$, and $\dot{\theta}$ denote velocities in the $\ell$, $Z$, and $\theta$ ``coordinate directions'',
  and
  \begin{align*}
  	\begin{pmatrix} a \\ b 
	\end{pmatrix} :=
	\begin{pmatrix}
    \hphantom{-}\cos(\theta) & \sin(\theta) \\
              - \sin(\theta) & \cos(\theta)
	\end{pmatrix}
	\begin{pmatrix}
		\dot{x} \\ \dot{y}
	\end{pmatrix}.
  \end{align*}
  These are moving frame coordinates, and the coordinates $(a,b)$ are sometimes called `pseudo-coordinates' since they are not induced by coordinates on $Q$.
  In terms of the local trivialization $Q = S \times \SE(2)$, we see that $(\ell,Z,\dot{\ell},\dot{Z})$ form standard induced coordinates on $S$ and $(\theta,x,y,\dot{\theta},a,b)$ are moving frame coordinates on $\SE(2)$ induced by left-trivialization of $\T \SE(2) \cong \SE(2) \times \mathfrak{se}(2)$.

  In these coordinates the left action of $\SE(2)$ on $TQ$ is naturally given by
  \begin{align*}
  	(\Theta, \Delta) \cdot
		\begin{pmatrix}
			\ell \\ Z \\ \theta \\ x \\ y \\ \dot{\ell} \\ \dot{Z} \\ \dot{\theta} \\ a \\ b
		\end{pmatrix} 
		= 
		\begin{pmatrix}
					\ell \\ Z \\ 
					\theta + \Theta \\
					\cos(\Theta) x - \sin(\Theta) y + \Delta_x \\
					\sin(\Theta) x + \cos(\Theta)y + \Delta_y \\
					\dot{\ell} \\ \dot{Z} \\
					\dot{\theta} \\ a \\ b
		\end{pmatrix}.
  \end{align*}
  We can immediately see that the quotient projection $\Pi$ merely projects out the $\theta,x,y$ coordinates,
  i.e.
   \begin{equation*}\label{eq:Pi}
  	\Pi(\ell, Z, \theta, x, y, \dot{\ell}, \dot{Z}, \dot{\theta}, a, b)
    =  (\ell, Z,               \dot{\ell}, \dot{Z}, \dot{\theta}, a, b).
  \end{equation*}
  \begin{rmk}
  On the open subset of $TQ$ where the coordinates $(\ell,Z,\theta,x,y)$ are valid,
  the map $(\ell,Z,\theta,x,y,\dot{\ell},\dot{Z},\dot{\theta},\dot{x},\dot{y}) \mapsto (\dot{\theta},a,b)$ is a principal connection
  if we identify $(\dot{\theta},a,b)$ as an element of $\mathfrak{se}(2)$.
  However, unlike the mechanical or Stokes connections, this map is not derived from physical properties of the system.
  It is merely derived from a non-canonical choice of coordinates that locally trivialize the principal bundle $Q$.
  \end{rmk}
  Recall that $TQ$ is a vector bundle over $Q$.
  The coordinates of $Q$ can be given by $(\ell,Z,\theta,x,y)$ and fibers of $TQ$ are parametrized by the coordinates $(\dot{\ell},\dot{Z},\dot{\theta},a,b)$.
  Similarly,
  the principal bundle $P$ is a vector bundle with base coordinates $(\ell,Z)$ and fibers coordinates $(\dot{\ell},\dot{Z},\dot{\theta},a,b)$.
  We see that $\Pi:TQ \to P$ is linear in the fiber coordinates (in fact it is the identity on the fibers with respect to these coordinates), and we could say that $P$ inherits the vector bundle structure of $TQ$ through the map $\Pi$.
  We denote by $P^*$ the vector bundle dual to $P$; this dual vector bundle will come into play shortly.
  
  Now that we understand $P$,
  we wish to assert the existence of a unique dynamical system on $P$
  which is consistent with the dynamical system on $TQ$ given by~\eqref{eq:full-EoM}.
  
  Note that~\eqref{eq:full-EoM} is written in terms of the total potential energy $U$ and the total dissipative force $F$.
  We observe that $U$ is $\SE(2)$ invariant because $U_g = U_g(Z)$, $U_\np = U_\np(Z)$, and $U_\s = U_\s(\ell)$.
  As a result, there exists a unique reduced potential $\widehat{U}\colon \SE(2) \backslash Q \to \R$ such that $U = \widehat{U} \circ \pi$.
  It is easy to believe that the differential $\d U: Q \to \T^*Q$ which appears in~\eqref{eq:full-EoM} must be $\SE(2)$ invariant as well;
  to understand this invariance we must consider how $\SE(2)$ acts upon $\T^*Q$.
  
    In a natural sense, the left action of $\SE(2)$ on $TQ$ induces a right action on $\T^*Q$.
    In standard Cartesian coordinates for the masses $(x_1,y_1,z_1,\dots,x_4,y_4,z_4)$ we may consider
  the fiber coordinates $(p_{x_1},p_{y_1},p_{z_1},\dots,p_{x_4},p_{y_4},p_{z_4})$ on $\T^*Q$,
  in which case the action is given by 
  \begin{align*}
		(\Theta, \Delta)^* \cdot
		\begin{pmatrix}
			x_i \\ y_i \\ z_i \\
			p_{x_i} \\ p_{y_i} \\ p_{z_i}
		\end{pmatrix}
		=
		\begin{pmatrix}
			\hphantom{-}\cos(\Theta) (x_i -\Delta_x) + \sin(\Theta) (y_i - \Delta_y) \\
			          - \sin(\Theta) (x_i -\Delta_x) + \cos(\Theta) (y_i - \Delta_y) \\
			z_i \\
			\hphantom{-}\cos(\Theta) p_{x_i} + \sin(\Theta) p_{y_i} \\
			          - \sin(\Theta) p_{x_i} + \cos(\Theta) p_{y_i} \\
			p_{z_i}
		\end{pmatrix}
  \end{align*}
  for $i=1,\dots,4$.  
  We may also consider this action in terms of the coordinates $(\ell,Z,\theta,x,y,p_\ell , p_Z, p_\theta, \bar{a} , \bar{b} )$
  where $(p_\ell,p_Z,p_\theta,\bar{a},\bar{b})$ are fiber coordinates conjugate to
  the fiber coordinates $(\dot{\ell},\dot{Z},\dot{\theta},a,b)$ on $TQ$.
  The $\SE(2)$ action on $\T^*Q$ is expressed in these coordinates as
  \begin{align*}
		(\Theta , \Delta )^* \cdot
		\begin{pmatrix}
			\ell \\ Z \\ \theta \\ x \\ y \\ p_\ell \\ p_Z \\ p_\theta \\ \bar{a} \\ \bar{b}
		\end{pmatrix}
		=
		\begin{pmatrix}
			\ell \\ Z \\ \theta - \Theta \\ \hphantom{-} \cos(\Theta) (x-\Delta_x) + \sin(\Theta) (y-\Delta_y) \\ -\sin(\Theta) (x-\Delta_x) + \cos(\Theta) (y-\Delta_y) \\ p_\ell \\ p_Z \\ p_\theta \\ \bar{a} \\ \bar{b}
		\end{pmatrix}	
  \end{align*}
  
  We say that $\d U$ is invariant if
  $(\Theta, \Delta)^* \cdot \d U( (\Theta , \Delta) \cdot q) = \d U(q)$
  for any $(\Theta,\Delta) \in \SE(2)$ and $q \in Q$.
  In other words, $\d U$ is invariant if the following diagram commutes
  \[
  \begin{tikzcd}
  \T^*Q & \arrow{l}{g^*} \T^*Q \\
  Q \arrow{u}{\d U} \arrow{r}{g} & Q  \arrow{u}{\d U}
  \end{tikzcd}
  \]
  for any $g \in \SE(2)$.
    In $(\ell,Z,x,y,\theta,p_\theta,\bar{a},\bar{b})$ coordinates $\d U:Q \to \T^*Q$ takes the form
  \begin{align*}
    \d U(\ell,Z,x,y,\theta) = \left(\ell,Z,x,y,\pder{U}{\ell}, \pder{U}{Z} , 0 ,0 \right).
  \end{align*}
  As $U$ is only a function of $\ell$ and $Z$ we see that $\pder{U}{\ell}$ and $\pder{U}{Z}$
  are only functions of $\ell$ and $Z$ as well.
  The group $\SE(2)$ acts trivially on the variables $\ell$ and $Z$
  and therefore we find
  \begin{align*}
    \d U( (\Theta , \Delta) \cdot (\ell,Z,x,y,\theta) ) =
		\begin{pmatrix}
			\ell \\ Z \\
      \theta + \Theta \\
			\cos(\Theta) x-\sin(\Theta)y \\
			\sin(\Theta) x+\cos(\Theta)y \\
			\partial U / \partial \ell \\
			\partial U / \partial Z  \\
			 0 \\ 
			 0
		\end{pmatrix}
  \end{align*}
  Applying $(\Theta, Z)$ to this we indeed verify that $\d U$ is $\SE(2)$ invariant.
  This invariance implies the existence of a unique map $\widehat{\d U}: \SE(2) \backslash Q \to P^*$,
  explicitly given in $P^*$ coordinates $(\ell,Z,p_\ell,p_Z,\bar{a},\bar{b})$ by
  \begin{align*}
    \widehat{\d U}(\ell,Z) = \left(\ell , Z , \pder{U}{\ell} , \pder{U}{Z} , 0 , 0 \right).
  \end{align*}

  Now that we have verified the invariance of $\d U$, we must do the same for the dissipative force $F:TQ \to \T^*Q$.
  If $F$ is invariant under the $\SE(2)$ action, then we should find that $(\Theta , \Delta)^* \cdot F( (\Theta, \Delta)  \cdot (q,\dot{q})) = F(q,\dot{q})$
  for all $(\Theta, \Delta) \in \SE(2)$ and $(q,\dot{q}) \in TQ$.
  In other words, $F$ is invariant if the following diagram commutes
   \[
  \begin{tikzcd}
  \T^*Q & \arrow{l}{g^*} \T^*Q \\
  TQ \arrow{u}{F} \arrow{r}{g} & TQ  \arrow{u}{F}
  \end{tikzcd}
  \]
  for any $g \in \SE(2)$.
  Intuitively, it is obvious that $F_\s$ and $F_\db$ are invariant
  because they are only functions of $\ell,Z,\dot{\ell}$ and $\dot{Z}$, upon which $\SE(2)$ acts trivially.
  The force $F_\ns$ is more subtle to analyze.
  We find that
  \begin{align*}
    &F_\ns ( (\Theta, \Delta) \cdot (q,\dot{q}) ) \\
    &\quad = c_\ns \sum_{i=1}^4 \chi'(z_i) \Big[
        \big( \cos(\Theta) \dot{x}_i - \sin(\Theta) \dot{y}_i \big) \d x_i
       +\big( \sin(\Theta) \dot{x}_i + \cos(\Theta) \dot{y}_i \big) \d y_i \Big]
  \end{align*}
  so that
  \begin{align*}
	&(\Theta, \Delta)^* \cdot F_\ns \big( (\Theta, \Delta) \cdot (q,\dot{q}) \big) \\
	&\quad = c_\ns \sum_{i=1}^4 \chi'(z_i)
     \begin{aligned}[t]
 \Big(&\big[   \cos(\Theta) (\cos(\Theta) \dot{x}_i - \sin(\Theta) \dot{y}_i)
             + \sin(\Theta) (\sin(\Theta) \dot{x}_i + \cos(\Theta) \dot{y}_i) \big] \d x_i \\
	   +&\big[{-}\sin(\Theta) (\cos(\Theta) \dot{x}_i - \sin(\Theta) \dot{y}_i)
             + \cos(\Theta) (\sin(\Theta) \dot{x}_i + \cos(\Theta) \dot{y}_i) \big] \d y_i
 \Big)\end{aligned} \\
	&\quad = c_\ns \sum_{i=1}^4 \chi'(z_i) \big(\dot{x}_i \d x_i + \dot{y}_i \d y_i \big) = F_\ns(q,\dot{q}).
  \end{align*}
  Thus $F_\ns$ is $\SE(2)$ invariant and therefore the total force $F$ is $\SE(2)$ invariant.
  An equivalent statement of $F_\ns$ being invariant would be that $\SE(2)$ acts by isometries with respect to the metric $\nu$.
  In any case, invariance of $F$ implies the existence of a unique map $\widehat{F}:P \to P^*$
  such that $ \langle \widehat{F}( \Pi( q,\dot{q}) ) , \Pi(q,v) \rangle = \langle F(q,\dot{q}) , (q,v) \rangle$ for any $(q,\dot{q}) , (q,v) \in TQ$.
  
  If we let $\xi = (\dot{\ell},\dot{Z},\dot{\theta},a,b)$ denote the fiber coordinates of $P$, we find that $\widehat{F}$ is of the form
  \begin{align*}
    \widehat{F}( \ell , Z , \xi )_i= -\hat{\nu}_{ij}(\ell , Z) \cdot \xi^j
  \end{align*}
  for some positive definite quadratic form $\hat{\nu}(\ell,Z)$ which is linearly related to $\nu(q)$  by an outer automorphism.
  In particular, $\hat{\nu}(s)$ is related to $\nu(q)$ by $\nu(q) (v_q,w_q) = \hat{\nu}(s) ( D_q \Pi(v_q) , D_q \Pi(w_q) )$.
  Locally, we may write $\nu(q)$ and $\hat{\nu}(s)$ as matrices, and the above relation takes the form of $\nu(q) = [D_q\Pi]^T \hat{\nu}(s) [D_q\Pi]$.
  The same analysis applied to the metric $k$ yields a fiber-wise quadratic form $\hat{k}$ on $P$ whose components $\hat{k}_{ij}(\ell,Z)$ only depend on the
  shape variables $(\ell,Z)$.
  The reduced Lagrangian $\widehat{L}:P \to \mathbb{R}$ can now be defined by the relation $\widehat{L} \circ \Pi = L$ and takes the form
  \begin{align*}
  	\widehat{L}(\ell,Z,\dot{\ell},\dot{Z},\dot{\theta},a,b) = \frac{1}{2} \hat{k}_{ij}(\ell,Z) \xi^i \xi^j - \widehat{U}(\ell,Z)
  \end{align*}
  If we group the coordinates as $s = (\ell,Z)$, $\dot{s} = (\dot{\ell},\dot{Z})$, and $\eta = (\dot{\theta},a,b) \in \mathfrak{se}(2)$, then 
  we may define the block-structure for $\hat{\nu}(s)$ given by
  \begin{align*}
  	\hat{\nu}(s) \cdot (\dot{s} , \eta) =
		\begin{bmatrix}
			\hat{\nu}_{ss} (s) & \hat{\nu}_{s\eta}(s) \\
			\hat{\nu}_{\eta s} (s) & \hat{\nu}_{\eta\eta}(s)
		\end{bmatrix}
		\begin{bmatrix}
			\dot{s} \\ 
			\eta
		\end{bmatrix}.
  \end{align*}
  The reduced equation of motion are then given by
  \begin{align}
    &\frac{\d}{\d t} \left( \pder{ \widehat{L}}{\dot{s}}(s,\dot{s},\eta) \right) - \pder{\widehat{L}}{s} (s,\dot{s},\eta)= -\hat{\nu}_{ss}(s) \, \dot{s} - \hat{\nu}_{s\eta}(s) \, \eta \label{eq:horizontal_LP} \\
    &\frac{\d}{\d t} \left( \pder{ \widehat{L}}{\eta}(s,\dot{s},\eta) \right) -\ad^*_\eta \left( \pder{\widehat{L}}{\eta} (s,\dot{s},\eta) \right)= -\hat{\nu}_{\eta s}(s) \, \dot{s} - \hat{\nu}_{\eta\eta}(s) \, \eta \label{eq:vertical_LP}
  \end{align}
  where $\ad^*_\eta: \mathfrak{se}(2)^* \to \mathfrak{se}(2)^*$ denotes the coadjoint action\footnote{
  	There are two conventions for the coadjoint action, and they differ by a minus sign.
	In this article, $\ad^*_\xi$ is defined as the dual of $\ad_\xi: \mathfrak{se}(2) \to \mathfrak{se}(2)$.
  }	
  of $\eta \in \mathfrak{se}(2)$ on $\mathfrak{se}(2)^*$.
  The appearance of the coadjoint action arises from the fact that $\eta$ is the $\mathfrak{se}(2)$ component of the local trivialization $TQ \cong TS \times \SE(2) \times \mathfrak{se}(2)$
  where $S = \SE(2) \backslash Q$.
  This local trivialization induces moving frame coordinates, and the equations of motion are altered.
  A description of Euler--Lagrange equations in moving frame coordinates is provided in~\cite[Sect.~1.4]{CuDuSn2010}.
  Additionally, an explicit derivation which explains the appearance of the coadjoint action is given in~\cite[Cor.~1.4.7]{CuDuSn2010}.
  The derivation here would be the same.  As the equations of motion are $\SE(2)$ invariant, we see that we can quotient out the $\SE(2)$ component and write them as equations on $P \cong TS \times \mathfrak{se}(2)$.
  Alternatively, we can view equations~\eqref{eq:horizontal_LP} and~\eqref{eq:vertical_LP} as an instance of Hamel's equations~\cite{BlochMarsdenZenkov2009} or a local version the Lagrange--Poincar\'e equations with an external force~\cite{MarsdenScheurle1993,CeMaRa2001}.

  Since these equations do not depend on $(\theta,x,y)$ anymore, they can be interpreted as living on the reduced space $P$.
  In summary, we observe $21$ degrees of freedom in~\eqref{eq:vertical_LP} and~\eqref{eq:horizontal_LP}  rather than $24$ degrees of freedom expressed in~\eqref{eq:full-EoM}.

\subsection{Linearizations about equilibria}
Let $q_* = (\ell_*,\theta_*,Z_*,x_*,y_*) \in Q$ be such that $\d U(q_*) = 0$.
Then $(q_*,0) \in TQ$ is an equilibrium point of the equations of motion~\eqref{eq:full-EoM}.
We can therefore consider the linearized equations over $(q_*,0)$ with respect to
the coordinates $(q,\dot{q}) = ( (\ell,Z,\theta,x,y) , (\dot{\ell},\dot{Z},\dot{\theta},\dot{x},\dot{y}) )$ on $TQ$.
It is a well-known result of the theory of linear oscillations, that the linearized system takes the form of a damped harmonic oscillator,
\[
  \frac{\d}{\d t} \begin{bmatrix} q \\ \dot{q} \end{bmatrix}
  = \begin{bmatrix} 0 & I \\ -\kappa & -\nu_* \end{bmatrix}
    \begin{bmatrix} q \\ \dot{q} \end{bmatrix}
\]
where $\kappa, \nu_* = \nu(q_*) \in \R^{12 \times 12}$ are positive (semi-)definite matrices given by
\[
  \kappa_{ij} := \left. \frac{ \partial^2 U}{\partial q^i \partial q^j } \right|_{q = q_*}
\]
and~\eqref{eq:viscosity_matrix}, respectively.
In particular, $\nu$ is the local manifestation of the dissipation force, and $\kappa$ represents the lowest order Taylor approximation of the potential energy at $q_*$~\cite{QEP}.

The principal bundle projection on $TQ$ is locally given by~\eqref{eq:Pi}
and the Jacobian of $\Pi$ at $(q_*,0)$ is locally given by the matrix
\begin{equation}\label{eq:harmonic_oscillator}
  D_{(q_*,0)} \Pi = \begin{bmatrix} \pr & 0 \\ 0 & \lambda \end{bmatrix}.
\end{equation}
where $\pr$ denotes the linear projection sending $(\ell,Z,\theta,x,y)$ to $(\ell,Z)$, and $\lambda$ is the linear isomorphism
which sends $(\dot{\ell},\dot{Z},\dot{\theta},\dot{x},\dot{y})$ to $(\dot{\ell},\dot{Z},\dot{\theta},a,b)$ by rotating $(\dot{x},\dot{y})$ by an angle of $-\theta_*$, i.e.\ the change of frame over $q_*$.
Under certain reasonable assumptions (see Assumption~\ref{ass:stability-assumptions} on page~\pageref{ass:stability-assumptions}), the reduced potential energy $\widehat{U}$ has a non-degenerate minimum which corresponds to the crawler resting motionless on the ground.
In this case we can verify that the kernel of $\kappa$ is the space generated by the action of $\SE(2)$, i.e.\ by translating and rotating along the ground.
Mathematically, this means
\begin{align*}
	\kernel(\kappa) = \operatorname{span} \left( \left.\pder{}{\theta} \right|_{q_*},\left.\pder{}{x} \right|_{q_*},\left.\pder{}{y} \right|_{q_*}\right).
\end{align*}
From~\eqref{eq:harmonic_oscillator} one can verify that
\begin{align*}
	\kernel( D_{(q_*,0)} \Pi ) = \kernel(\kappa).
\end{align*}
Therefore, by Proposition~\ref{prop:linearization}, the linearization of the reduced system on $P$ about the equilibrium $(s_*,0) = (\ell_*,Z_*,0) = \Pi(q_*,0)$ is given by
\begin{equation}\label{eq:linear-reduced}
	\frac{\d}{\d t} \begin{bmatrix} s \\ \xi \end{bmatrix}
  = \begin{bmatrix} 0 & \pr \\  - \hat{\kappa}\,\pr^T& - \hat{\nu}_* \end{bmatrix}
    \begin{bmatrix} s \\ \xi \end{bmatrix},
\end{equation}
where $\hat{\nu}_* := \lambda \nu_* \lambda^{T}, \hat{\kappa}:= \lambda \kappa \lambda^{T}$, and
where we have used the right inverse
\[
	(D_{(q_*,0)} \Pi)^{-1}_{\rm right}
	= \begin{bmatrix} \lambda^{T} \pr^T & 0 \\ 0 & \lambda^{T} \end{bmatrix}.
\]
Of course, \eqref{eq:linear-reduced} is nothing but the linearization of the reduced equations of motion about the equilibrium $(\ell_*,Z_*, 0) \in \SE(2)\backslash TQ$.
The matrix $\hat{\kappa}$ is an outer transformation of the Hessian of the reduced potential energy $\widehat{U} = \widehat{U}(\ell,Z)$.

\subsection{Stable equilibria}\label{sec:stable-equilibria}

  It is easy to intuit the existence of a stable equilibrium which corresponds to a stationary crawler resting on the ground.
 Such a point in phase space would be merely a single element of an entire $\SE(2)$-orbit of equilibria obtained by translating
 and rotating the crawler along the ground.
 Therefore, these equilibria can only be \emph{marginally} stable at best, as the vector field vanishes along the direction of this symmetry.
 However, it is possible that this $\SE(2)$-orbit projects to a (\emph{robustly}) stable equilibrium in the reduced system (in the sense of Definition~\ref{def:stable-equilibrium} below).
 We therefore turn to the reduced system and identify reasonably general conditions under which there exists a configuration $s_* \in \SE(2) \backslash Q$ which is a non-degenerate minimum of $\widehat{U}$.
 Then we apply Proposition~\ref{prop:stability} to conclude that $(s_*,0) \in P$ is a stable equilibrium.
 There exist a few competing definitions of stability, so to be completely unambiguous about what we mean, let us define
 \begin{defn}[Stable equilibrium]
   \label{def:stable-equilibrium}
   Let $\dot{x} = f(x)$ denote a dynamical system on a manifold $M$. Then we call $x_* \in M$ a \emph{robustly stable equilibrium} if $f(x_*) = 0$ and the spectrum of $D f(x_*)$ lies strictly left of the imaginary axis.
 \end{defn}
 This definition is to be seen in contrast to weaker notions such as \emph{marginal stability} wherein eigenvalues may lie on the imaginary axis.
  In particular, a robustly stable equilibrium is a hyperbolic fixed point which (locally) attracts solution curves at an exponential rate.

  To find a robustly stable equilibrium in our system, we make the following assumption:
  \begin{assn}\label{ass:stability-assumptions}~
    The rest lengths $\bar{\ell}_{ij}$ of the springs form a non-degenerate tetrahedron.
  \end{assn}

  We shall formulate the precise results that lead towards the existence of a robustly stable equilibrium in the propositions below and indicate the ideas of the proofs; the details can be found in Appendix~\ref{app:stability-proofs}.

  \begin{prop}\label{prop:potential-minimum}
    Under Assumption~\ref{ass:stability-assumptions},
    for sufficiently large $\kappa_\s$ and $\kappa_{\np}$ there exists a (local) minimum $s_* \in \SE(2) \backslash Q$ of the reduced potential $\widehat{U}$.
    This minimum is non-degenerate in the sense that the Hessian, $\hat{\kappa}$, of $\widehat{U}$ at $s_*$ is positive definite.
  \end{prop}
  
  For reasons which will be clear soon, we must have a guarantee that one mass of the equilibrium configuration has a larger $z$ coordinate than the others.
  Such a guarantee requires that the springs be sufficiently stiff to support the weight.
  This minimum spring stiffness, $\kappa_\s$, will implicitly depend on how close to degeneracy the tetrahedron formed by the rest lengths is; this ensures that the actual lengths, $\ell_{ij}$, of the energy-minimizing configuration form a non-degenerate tetrahedron.
  The idea now is to search for a configuration where the masses $1,2$ and $3$ `rest on the ground' and $4$ has coordinate $z_4 > 0$ raised above the influence of the ground potential.
  We view this as a singular perturbation problem: when the stiffnesses $\kappa_\s, \kappa_{\np}$ are infinite, then the solution is trivially the rigid tetrahedron with side
  lengths $\bar{\ell}_{ij}$ and resting on the ground, i.e.\ $z_1=z_2=z_3 = 0$.
  By rescaling, we turn it into a regular perturbation problem and apply the implicit function theorem to find a slightly perturbed stable configuration for large but finite $\kappa_\s, \kappa_{\np}$.

Secondly, the viscous friction is non-degenerate.
As a preliminary result to proving hyperbolic attractivity of the fixed point in Proposition~\ref{prop:stability}, we prove
\begin{prop} \label{prop:nondegenerate}
  The matrix $\hat{\nu}$ is positive definite on the vector bundle fiber of $P$ above $s_*$,
  where $s_*$ is the minimum found in Proposition~\ref{prop:potential-minimum}.
\end{prop}
The proof is given in Appendix~\ref{app:stability-proofs}.

Together with the nature of the minimum $s_*$ of $\widehat{U}$, this provides all prerequisites for the following
\begin{prop}\label{prop:stability}
  Let $s_* \in \SE(2) \backslash Q$ be a non-degenerate minimum of $\widehat{U}$, that is, $\d\widehat{U}(s_*) = 0$ and its Hessian $\hat{\kappa}$ is positive definite.
  Then $(s_*,0) \in P$ is a robustly stable equilibrium for the reduced system.
\end{prop}

The idea is that if no friction were present, then starting close to the stable equilibrium $(s_*,0)$ in phase space, the motion would be oscillatory.
Since the friction force is non-degenerate by Proposition~\ref{prop:nondegenerate}, the energy will decay asymptotically, sending the system to a standstill at $(s_*,0)$. We prove that this decay towards $(s_*,0)$ is exponential.
Note that any $q \in Q$ such that $\pi(q) = s_*$ produces an equilibrium $(q,0) \in TQ$ for the unreduced system.
However, any such $q$ is \emph{not} a robustly stable equilibrium (it is only marginally stable).

\subsection{Time-periodic perturbations} \label{sec:time_periodic}

  Given a dynamical system $\dot{x} = f(x)$ on a manifold $M$ with a robustly stable equilibrium $x_* \in M$, one can embed the system into a time-periodic augmented phase space $S^1 \times M$ by using the vector field $(\dot{t}, \dot{x}) = (1,f(x))$.
  Then the trajectory $\gamma_0(t) = (t,x_*) \in S^1 \times M$ is a limit cycle for the system on $S^1 \times M$ which locally attracts at an exponential rate.
  The orbit $\Gamma_0 := S^1 \times \{ x_* \}$ is a compact normally hyperbolic invariant submanifold, and so the theorem on persistence of normally hyperbolic invariant manifolds~\cite{Fenichel1971,Hirsch77} applies.
  Specifically, given a sufficiently small\footnote{%
    To be more precise, the perturbation must be small in $C^1$ supremum norm. The Lagrange--d'Alembert vector field was already smooth (after application of a mollifier).
    Since we augmented the phase space with periodic time, these theorems also require the perturbation to be $C^1$ with respect to time.
    Note however that this can be relaxed to continuous~\cite[Remark~4.1]{Eldering2013} and possibly integrable dependence on time.%
  } time-periodic perturbation $f \mapsto f + \varepsilon g_t $, we can assert the existence of a \emph{persistent} limit cycle,  $\gamma_{\varepsilon}$, in a neighborhood of $\gamma_0$  (see also `The Averaging Theorem' in~\cite{GuckenheimerHolmes}).

  In the previous subsection, we found a robustly stable equilibrium in $P$.
  In this section, we will perturb this system by substituting time $T$-periodic lengths $\bar{\ell}_{ij}(t)$ for the constant rest lengths $\bar{\ell}_{ij}$.
  If these oscillations are small, we can expect to observe a $T$-periodic limit cycle, $(t, \hat{\gamma}(t))$, in the augmented phase space $S^1 \times P$.
  Thus $\hat{\gamma}(t) = (\ell(t) , Z(t) , \dot{\ell}(t), \dot{Z}(t), \dot{\theta}(t) , a(t) ,b(t))$ is a stable periodic trajectory of the original time-periodic system on $P$.
  However, if $\gamma(t)$ is a trajectory in $TQ$ which projects down to $\hat{\gamma}(t) \in P$, then it is generally not the case that $\gamma(t)$ is periodic.
  In particular, a periodic trajectory $\hat{\gamma} \subset P$ is the projection of many trajectories $(q ,\dot{q} )(t) \in TQ$ such that
\begin{equation}\label{eq:rel_periodic}
  (q,\dot{q})(t + T) = (\Theta, \Delta ) \cdot (q,\dot{q})(t),
\end{equation}
for some (fixed) element $(\Theta , \Delta) \in \SE(2)$.
  Trajectories which satisfy conditions such as~\eqref{eq:rel_periodic} are known as \emph{relatively periodic orbits}.
  A relatively periodic orbit $\gamma(t)$ emanating from an initial condition $\gamma(0) \in TQ$ will project down to a periodic orbit $\hat{\gamma}(t) = \Pi( \gamma(t))$ in $P$.
  Conversely, an orbit $\gamma(t)$ which projects down to a periodic orbit $\hat{\gamma}(t) = \Pi( \gamma(t) )$ in $P$ is necessarily a relatively periodic orbit in $TQ$.

  Moreover, if $\hat{\gamma}$ is a stable limit cycle in $P$, then the relatively periodic orbits in $TQ$ are stable as well, and marginally stable along the $\SE(2)$ orbits.
  In this case the orbits in $TQ$ are dubbed `relative limit cycles' in that they are relatively periodic and stable.
  For our system, the \emph{phase} $(\Theta, \Delta) \in \SE(2)$ corresponds to the translation and rotation of the crawled after one cycle.
  This completes the proof of all claims in our main theorem~\ref{thm:limit-cycle}.

The phase can be reconstructed from the periodic orbit $\hat{\gamma}(t) \in P$ in the following way.
\begin{thm}
Let $\hat{\gamma}(t) \in P$ be a $T$-periodic orbit
and let $\gamma(t) = \frac{\d q}{\d t}(t) \in TQ$ be such that $\hat{\gamma}(t) = \Pi(\gamma(t))$.
Then the phase of the relative periodic orbit, $\gamma(t)$,
is obtained by solving the initial value problem on $\SE(2)$ given by:
\begin{equation}\label{eq:reconstruction}
  \begin{cases}
    &\frac{\d \Theta}{\d t} = \dot{\theta}(t) \\
    &\frac{\d \Delta_x}{\d t} = \cos(\Theta) a(t) - \sin(\Theta) b(t) \\
    &\frac{\d \Delta_y}{\d t} = \sin(\Theta) a(t) + \cos(\Theta) b(t) \\
    &\Theta(0) = 0 \quad,\quad \Delta(0) = 0.
	\end{cases}
\end{equation}
The phase is $(\Theta(T) , \Delta(T) )$.
Alternately, we may write the reconstruction equation as a left-invariant ODE on $\SE(2)$ as $\dot{g} = g \cdot \eta(t)$
where $\eta(t) = (\dot{\theta},a,b)(t) \in \mathfrak{se}(2)$ and $g(t) \in \SE(2)$.
\end{thm}

\begin{proof}
	Let $\hat{\gamma}(t) = (\ell,Z,\dot{\ell},\dot{Z},\dot{\theta},a,b)(t) \in P$ be a $T$-periodic orbit.
	Let $\gamma(t) = (\ell,Z,\theta,x,y,\dot{\ell},\dot{Z},\dot{\theta},a,b)(t)$ be a relative periodic orbit in $TQ$
	which projects to $\hat{\gamma}$.  Assume $x(0) = y(0) = 0$.
	Moreover, we will assume $\gamma(t) \in TQ$ is the time derivative of a curve in $Q$.  That is:
	\begin{align*}
		&\frac{\d\ell}{\d t} = \dot{\ell} \quad,\quad \frac{\d Z}{\d t} = \dot{Z} \quad,\quad \frac{\d\theta}{\d t} = \dot{\theta} \\
		&\frac{\d x}{\d t} = \cos(\theta) a - \sin(\theta) b \quad,\quad
     \frac{\d y}{\d t} = \sin(\theta) a + \cos(\theta) b
	\end{align*}
	As $\SE(2)$ acts upon the $\theta,x,y$ coordinates freely and transitively,
	there must exist a unique curve $(\Theta , \Delta)(t) \in \SE(2)$
	such that
	\begin{align*}
		(\theta(t) , x(t) , y(t)) &= (\Theta , \Delta)(t) \cdot (\theta(0) , x(0) , y(0))  \\
		&= \left(\theta(0) + \Theta(t) , 
			\begin{pmatrix}
				\cos(\Theta(t)) x(0) - \sin(\Theta(t)) y(0) + \Delta_x(t) \\
				\sin(\Theta(t)) x(0) + \cos(\Theta(t)) y(0) + \Delta_y(t)
			\end{pmatrix} \right).
	\end{align*}
	From the equation for $\frac{\d\theta}{\d t}$ it is clear that $\frac{\d \Theta(t)}{\d t} = \dot{\theta}$.
	Upon taking the derivative of $x(t)$ in the above equation we find
	\begin{align*}
		\frac{\d}{\d t}
		\begin{pmatrix}
			x \\
			y
		\end{pmatrix}
		(t)		=
		\dot{\theta}(t)
		\begin{pmatrix}
			          - \sin(\Theta(t)) & -\cos(\Theta(t)) \\
			\hphantom{-}\cos(\Theta(t)) & -\sin(\Theta(t))
		\end{pmatrix}
		\begin{pmatrix}
			x(0) \\
			y(0)
		\end{pmatrix}
		+ \frac{\d \Delta}{\d t}(t)
	\end{align*}
	We first consider the case $\big(\theta(0),x(0),y(0)\big) = (0,0,0)$, so we can ignore the first term.
	Upon substitution of the equation for $\d x/\d t$ and $\d y/\d t$ into the previous line the claim follows.
	Note that we have derived the phase $(\Theta,\Delta)(T) \in \SE(2)$
	purely in terms of coordinate functions on $P$. More abstractly put, we have solved the left-invariant ODE
  \begin{equation*}\label{eq:phase-ODE}
    \dot{g} = g \cdot \eta(t), \qquad g(0) = \textrm{id},
  \end{equation*}
  with $g(t) = (\Theta,\Delta)(t) \in \SE(2)$.

  Next we let $h(t) = (\theta,x,y)(t) \in \SE(2)$ and consider the
  general case $h_0 = \big(\theta(0),x(0),y(0)\big) \neq \textrm{id}$.
  Again, since $\SE(2)$ acts freely and transitively on itself, there
  exists a unique curve $g(t)$ such that $h(t) = g(t) \cdot h_0$.
  Taking a time derivative and substituting $g(t) = h(t) h_0^{-1}$
  yields
  \begin{equation*}
    \dot{g} = \dot{h} h_0^{-1}
            = g (h h_0^{-1})^{-1} \dot{h} h_0^{-1}
            = g h_0 h^{-1} \dot{h} h_0^{-1}
            = g \cdot \textrm{Ad}_{h_0^{-1}} \big(\eta(t)\big),
  \end{equation*}
  since $\eta = h^{-1} \dot{h}$ and still with initial condition
  $g(0) = \textrm{id}$. The phase shift is given by
  \begin{equation*}
    g(T) = h_0^{-1} \cdot (\Theta,\Delta) \cdot h_0 \in \SE(2),
  \end{equation*}
  which is the original phase shift $(\Theta,\Delta)$ conjugated by
  the initial condition $h_0 \neq \textrm{id}$. Note that the initial
  condition $h_0$ multiplied the left-invariant vector field $\eta(t)$
  from the right, and hence modified it.
\end{proof}

  To compute the phase shift $(\Theta, \Delta)$, we have to integrate~\eqref{eq:reconstruction} over one cycle of a periodic orbit.
  The periodic orbit of interest to us is a persistent limit cycle,
  whose existence we can assert, but whose form is not known to us.
  Fortunately, the present system is simple enough to be studied in computer simulations, see  section~\ref{sec:numerical-sim}.
  The simulations we carried out revealed that the phase shift appears generically to be non-zero, but to depend on the perturbation size to second order.
  A heuristic explanation for this result can be given by the fact that `making a step' requires the combined variation of position \emph{and} velocity of the masses, leading to a quadratic dependence on the perturbation size.
  The variation in position is needed to displace the crawler's weight towards a leg and the variation in velocity to actually move the other leg(s).

  We shall now give a rigorous argument that $(\Theta, \Delta) \in \mathcal{O}(\varepsilon^2)$ with $\varepsilon$ the perturbation size parameter.
  First of all, since the fixed point $(s_*,0) \in \SE(2) \backslash TQ$ is hyperbolic, the perturbation of the limit cycle will scale linearly with $\varepsilon$ as well (this follows from smooth dependence of a NHIM on parameters).
  Let us denote the periodic orbit in $\SE(2) \backslash TQ$ by
  \begin{equation*}
    (s,\xi)_\varepsilon(t)
    = (\ell,Z,\dot{\ell},\dot{Z},\dot{\theta},a,b)_\varepsilon(t).
  \end{equation*}

  Recall that the potential forces (including the actuation forces) do not act along the group directions, hence the Lagrange--d'Alembert equations with respect to the associated moving frame coordinates $\eta = (\dot{\theta},a,b) \in \mathfrak{se}(2)$ reduced to~\eqref{eq:vertical_LP}:
  \begin{equation*}
    \frac{\d}{\d t} \pder{ \widehat{L}}{\eta} - \ad^*_\eta \Big( \pder{\widehat{L}}{\eta} \Big)
    = -\hat{\nu}_{\eta s}(s) \, \dot{s} - \hat{\nu}_{\eta\eta}(s) \, \eta.
  \end{equation*}
  When we integrate~\eqref{eq:vertical_LP} over a full period, we find that
  the first term on the left-hand side integrates to zero as it is
  the time derivative of a periodic function. We also note
  that the coordinates $\dot{s} = (\dot{\ell},\dot{Z})$ are induced velocity
  coordinates of periodic coordinates $s = (\ell,Z)$ and hence integrate
  to zero as well, and finally that the second term
  is quadratic in the velocities $\xi$.

  We now perform a Taylor expansion in $\varepsilon$ using notation
  \begin{align*}
    s_\varepsilon(t)
    &= s^{(0)}(t) + \varepsilon s^{(1)}(t) + \mathcal{O}(\varepsilon^2), \\
    \xi_\varepsilon(t)
    &= \xi^{(0)}(t) + \varepsilon \xi^{(1)}(t) + \mathcal{O}(\varepsilon^2),
  \end{align*}
  where $(s^{(0)},\xi^{(0)})(t) = (s_*,0)$ is the
  unperturbed fixed point and we note that
  $\xi_\varepsilon(t) \in \mathcal{O}(\varepsilon)$. Expanding the
  right-hand side term of~\eqref{eq:vertical_LP} yields
  \begin{equation}\label{eq:friction-1st-order}
    \hat{\nu}_{\eta s} (s(t)) \, \dot{s}(t) +  \hat{\nu}_{\eta \eta} (s(t)) \, \eta(t)
    = \varepsilon\, (\hat{\nu}_{\eta s}(s_*) \, \dot{s} +  \hat{\nu}_{\eta \eta} (s_*) \, \eta^{(1)}(t) ) + \mathcal{O}(\varepsilon^2)
  \end{equation}
  since $\xi^{(0)} = 0$. Finally we substitute the right hand side of~\eqref{eq:vertical_LP} with the right hand side of~\eqref{eq:friction-1st-order}
  and integrate over a period to yield
  at order $\varepsilon^1$
  \begin{equation*}
    0 = - \varepsilon \int_0^T \hat{\nu}_{\eta s}(s_*) \, \dot{s} +  \hat{\nu}_{\eta \eta} (s_*) \, \eta^{(1)}(t)  \,\d t.
  \end{equation*}
  As $\hat{\nu}(s_*)$ is constant and $\dot{s}$ is the
  time derivative of a periodic function, only the integral over
  $\eta^{(1)}(t)$ remains. Since $\hat{\nu}_{\eta\eta}(s_*)$ is
  simply $\hat{\nu}(s_*)$ restricted to the linear subspace
  spanned by the $\eta$ coordinates, it is still non-degenerate. As a
  result we conclude that
  \begin{equation}\label{eq:eta-integral}
    0 = \int_0^T \eta^{(1)}(t) \,\d t.
  \end{equation}
  To conclude that the first order perturbation of the phase shift is
  zero, we have to integrate the ODE $\dot{g} = g\cdot\eta(t)$ on
  $\SE(2)$ over a period. The Magnus expansion gives that
  \begin{equation*}
    g(T) = \operatorname{exp}\Big(
             \int_0^T \eta(t)\,\d t +
  \frac{1}{2}\int_0^T \int_0^{t_1} [\eta(t_1),\eta(t_2)] \,\d t_1 \,\d t_2 + \cdots
           \Big)
  \end{equation*}
  where further terms contain repeated commutator brackets.
  From~\eqref{eq:eta-integral} it follows that the first term in the
  exponent vanishes at order $\varepsilon$, while all further terms
  vanish at order $\varepsilon$ due to the appearance of (repeated) commutators
  of $\eta_\varepsilon(t) \in \mathcal{O}(\varepsilon)$. The upshot
  here is that the curvature (i.e.~non-Abelianness) of a group
  expressed by these commutators only contributes at higher orders of
  $\varepsilon$. This proves our claim that
  $(\Theta,\Delta) \in \mathcal{O}(\varepsilon^2)$.

  On the other hand, one can generically expect to see a non-zero
  phase shift at order $\varepsilon^2$, as corroborated by our
  numerical simulations in the next section. There are two
  contributing effects to this. Firstly, in the second order expansion
  \begin{equation*}\label{eq:friction-2nd-order}
    \hat{\nu}(s) \cdot \xi
    = \varepsilon  \, \hat{\nu}(s_*) \cdot \xi^{(1)}
     +\varepsilon^2\, \Big(
        \big(D_i \hat{\nu}(s_*) \cdot \xi^{(1)}\big) \, s^{(1)\,i}
               + \hat{\nu}(s_*) \cdot \xi^{(2)}\Big)
     + \mathcal{O}(\varepsilon^3)
  \end{equation*}
  the term $(D_i \hat{\nu}(s_*) \cdot \xi^{(1)})\,s^{(1)\,i}$
  integrates to a non-zero contribution over a period when $\hat{\nu}$
  depends non-trivially on the reduced configuration variables
  $s \in \SE(2) \backslash Q$. This can be viewed in contrast
  to~\cite{ChangSoo2013}, where damping induced self-recovery of a
  cyclic variable is studied, and hence a non-zero phase shift cannot
  occur. That setting assumes that $\hat{\nu}$ does not depend on the other
  variables.

  Secondly, the non-Abelianness of $\SE(2)$ allows for a non-zero
  contribution even when $\int_0^T \eta(t) \,\d t = 0$. This can be
  thought of as a holonomy defect due to curvature of the symmetry
  group; indeed the defect depends to second order on the path length,
  which is of order $\varepsilon$.

\section{Numerical simulations}\label{sec:numerical-sim}

  In this section we numerically compute trajectories to better understand this system.
  We first present results for a 2D walker with three masses in the $xz$-plane, see Figure~\ref{fig:2D-crawler}.
  In this case the phase shift is simply a translation $\Delta x \in \R$, but all other features of the model are still retained; at the end of the section we show results for a 3D simulation.

  In particular, for the 2D model, we consider the time dependent spring lengths
  \begin{align*}
	\bar{\ell}_{1} (t) &= 1 + \varepsilon \cos(\omega t)  \\
	\bar{\ell}_{2} (t) &= 1 - \varepsilon \sin\Big(\omega\Big(t - \frac{1}{2}\Big)\Big) \\
	\bar{\ell}_{3} (t) &= 3 - \bar{\ell}_{1}(t) - \bar{\ell}_{2}(t)
  \end{align*}
  where $\omega = 2\pi$ and we vary the amplitude $\varepsilon > 0$.  Additionally we use the parameters:
  $\kappa_\np = 10$, $c_\ns = 10$, $\kappa_\s = 10$, $c_\db = 5$, and $c_\s = 10$.

 \begin{figure}[p]
    \centering
    \input{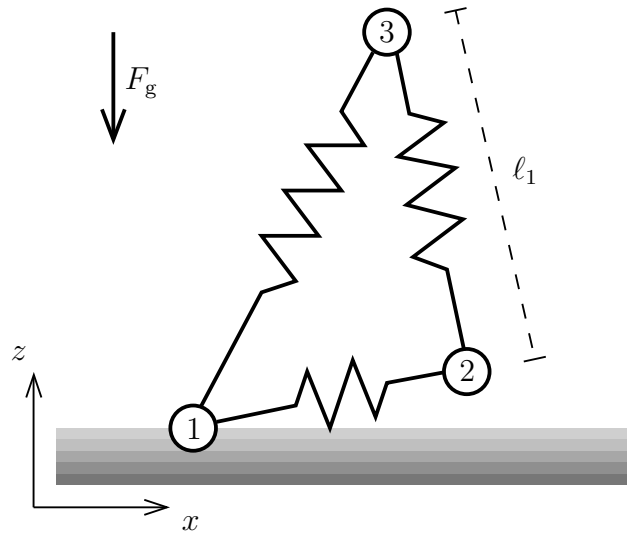}
    \caption{The 2D crawler}
    \label{fig:2D-crawler}
 \end{figure}

 \begin{figure}[p] 
    \centering
    \includegraphics[width=5.5in]{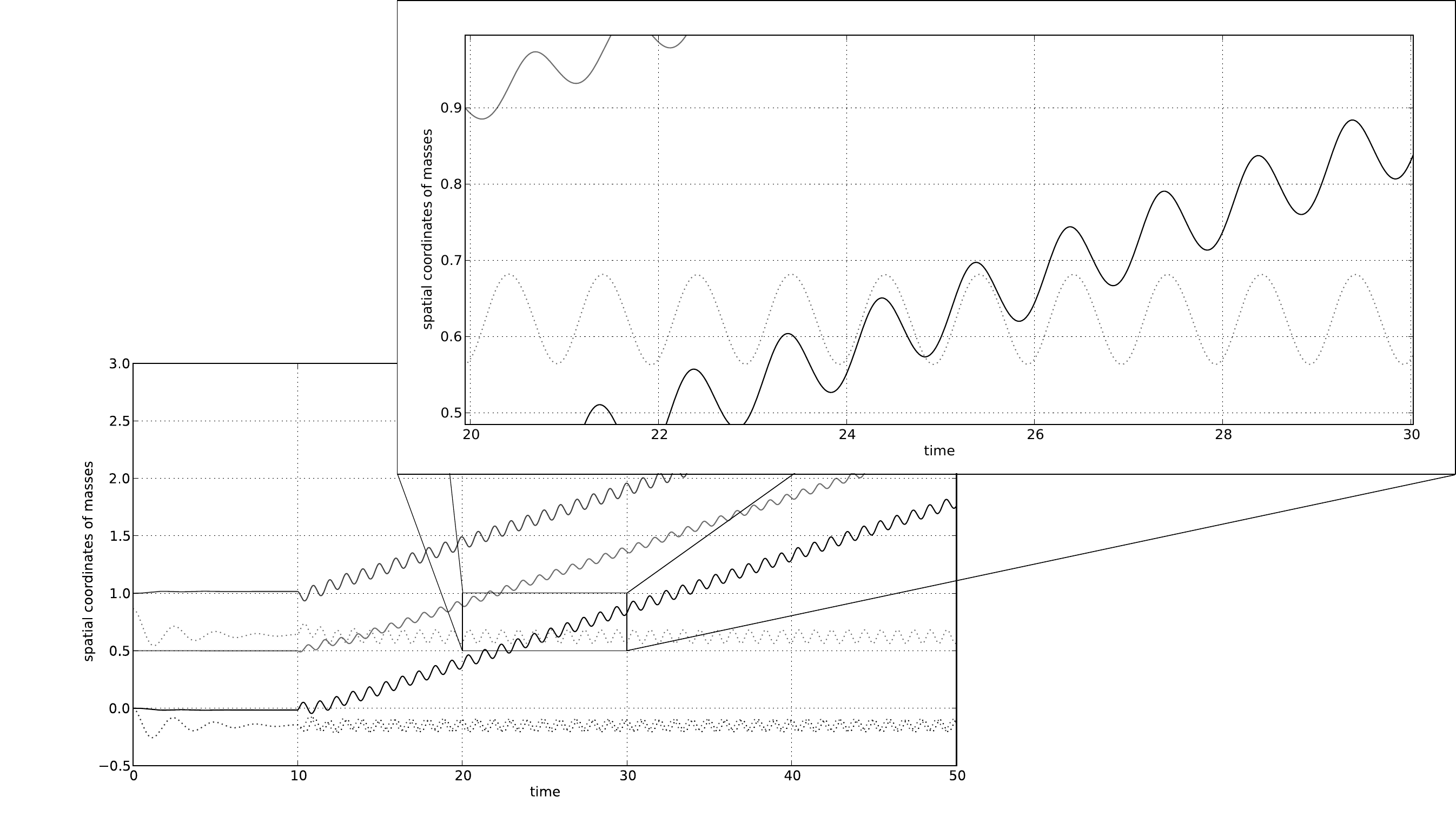}
    \caption{This plot depicts the $x$ coordinates (bold lines) and the $z$ coordinates (thin lines) of the three masses for a trajectory where $\varepsilon = 0.5$.  The system is activated at $t=10$.}
    \label{fig:evolution}
 \end{figure}

\begin{figure}[p] 
   \centering
   \includegraphics[width=4in]{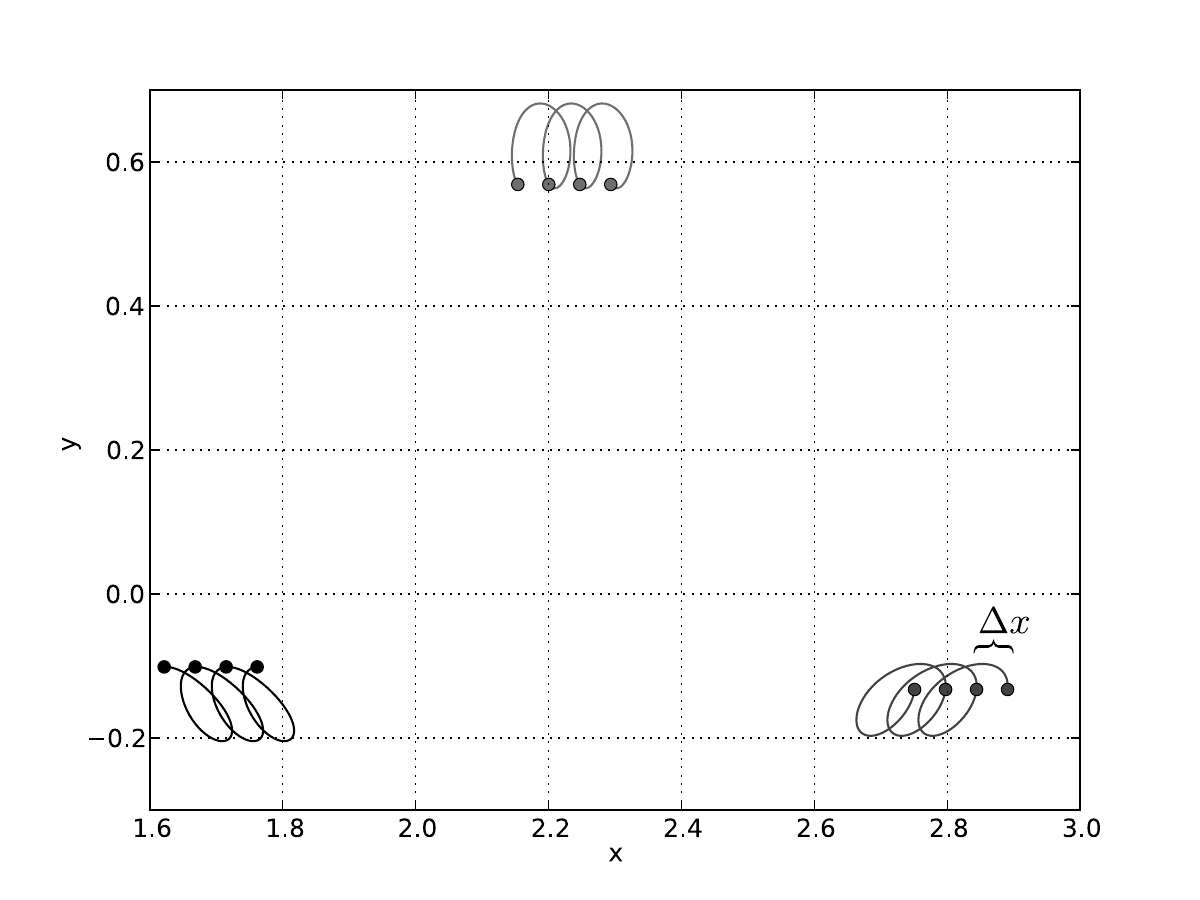}
   \caption{Depicted are the trajectories of the masses in space over the final three periods plotted in figure~\ref{fig:evolution}.  Above the trajectory of the bottom right mass we have indicated the phase shift of $\Delta x = 0.046$.}
   \label{fig:mass_traj}
\end{figure}

  To test our theory we allow the system $10$ seconds of inactivity (i.e.~$\varepsilon = 0$) so that the system settles towards an equilibrium.
  Then, at $t=10$ we set $\varepsilon = 0.5$.
  The system appears to converge to a relatively periodic orbit after a few periods, see Figure~\ref{fig:evolution}.
  This relatively periodic orbit exhibits a phase shift of $\Delta x = 0.046$, and so we observe a steady drift in the positive $x$-direction.
    We observe that both the $x$ and $z$ coordinates oscillate with angular frequencies of $2 \pi$, as predicted by our analysis in Section~\ref{sec:time_periodic}, i.e.\ the period of the relative limit cycle is identical to the period of the perturbation.
  To further illustrate this relatively cyclic behavior we have plotted the locations of the masses over three time-periods in Figure~\ref{fig:mass_traj} where one can clearly see how each period is identical to the previous period up to the constant shift $\Delta x = 0.046$.  Finally, this value of $\Delta x$ was observed to be robust to small but randomly chosen changes in the initial conditions.  This is in agreement with the theory that $\Delta x$ is ultimately a function of the time dependent lengths $\bar{\ell}_k(t)$ only, implicitly defined through the phase reconstruction formula~\eqref{eq:reconstruction}.

  Although we do not have a proof that $\Delta x$ is generically non-zero, a few trial perturbations all yielded non-zero $\Delta x$ values.
  The simulations do support the claim that the first variation of $\Delta x$ with respect to the perturbation is zero, while the second variation is non-zero.
  In particular we have calculated trajectories for various $\varepsilon$'s, and computed the quantities
  \[
  	p_k = \frac{ \log \left| \Delta x_k \right| - \log\left| \Delta x_{k-1} \right| }{ \log( \varepsilon_k) - \log(\varepsilon_{k-1}) },
  \]
  to detect the scaling of $\Delta x$ with the perturbation size.
  If $\Delta x$ is proportional to $\varepsilon^2$ then we should find that $p_k \approx 2$.  The results are summarized in Table~\ref{tab:deltax}.

\begin{table}[htbp]
  \begin{tabular}{|c|l|c|}
  	\hline
    $\varepsilon$ & $\quad\Delta x$ & $p$ \\
    \hline
    1    & 0.17870  & 1.9372 \\
    1/2  & 0.04666  & 1.9932 \\
    1/4  & 0.01172  & 1.9980 \\
    1/8  & 0.002934 & 1.9990 \\
    1/16 & 0.000734 & 2.0039 \\
    1/32 & 0.000183 &  N/A   \\
    \hline
  \end{tabular}
  \vspace{2ex}
  \caption{The values of $\Delta x$ for various perturbation sizes $\varepsilon$.}
  \label{tab:deltax}
\end{table}

Finally, Figure~\ref{fig:3D_curved} shows a simulation of a 3D walker.
Here, the trajectory is clearly curved due to the (very small) phase
shift having both a translational and rotational component. Phase
shifts that consist purely of either rotations or translations are
easily constructed by choosing the right symmetry for the
perturbation.

\begin{figure}[hptb]
   \centering
   \includegraphics[width=5in]{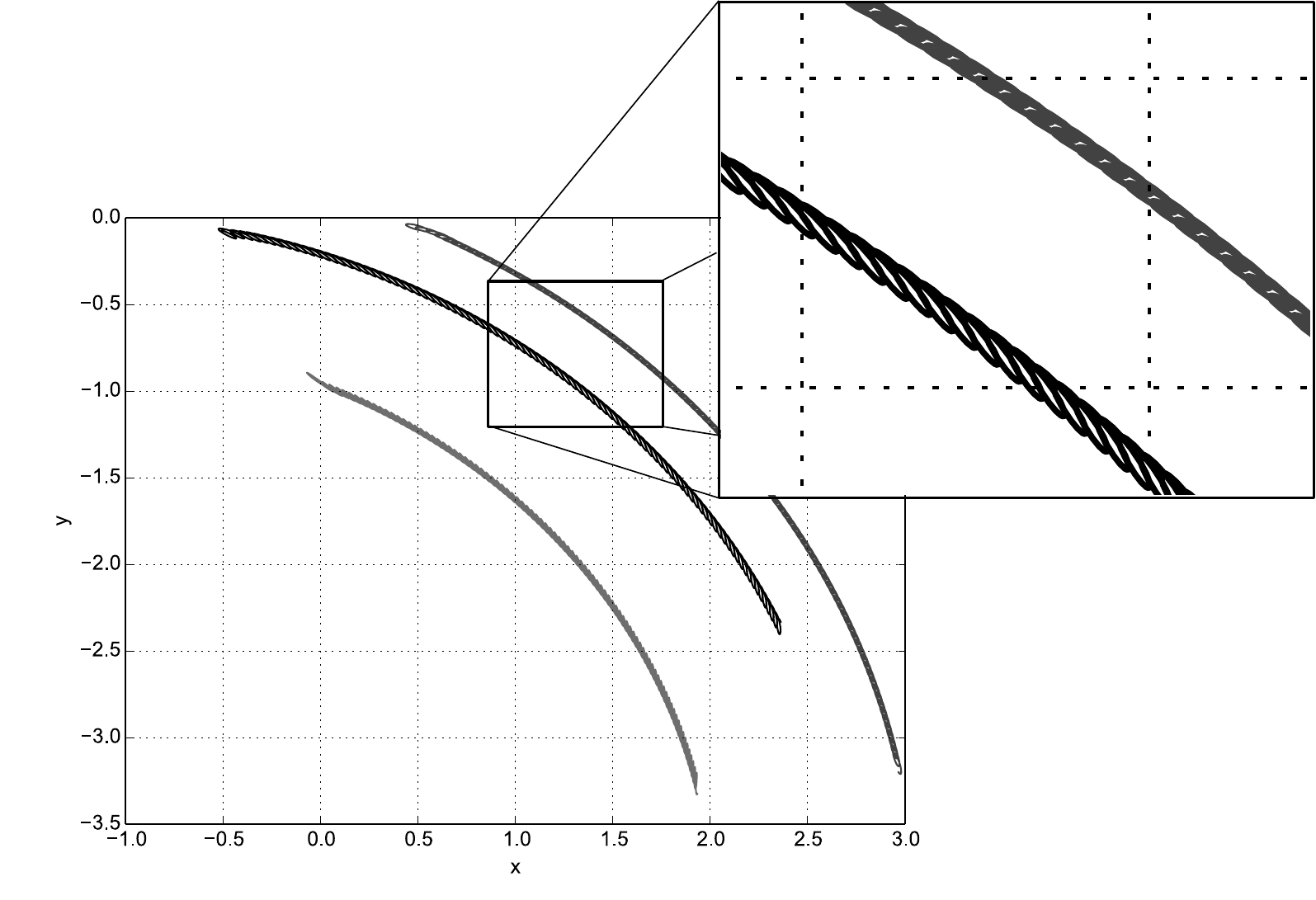}
   \caption{Depicted are the trajectories in the $xy$-plane of the three masses on the ground for the 3D crawler. The zoom box shows the single cycles, while the overall trajectories are clearly seen to curve.}
   \label{fig:3D_curved}
\end{figure}

\FloatBarrier

\section{Outlook \& conclusion}
In this paper we have shown that regularized models are capable of exhibiting behavior which resembles crawling, by constructing a model with a robust relative limit cycle.
  Such models are open to classical techniques in dynamical systems, and allow one to view crawling as a limit cycle in a reduced space, while the absolute motion manifests as a phase shift after reconstruction.
  These ideas are generic enough that it seems feasible to apply them to a range of other scenarios.

  Furthermore, the work suggests a number of follow-up questions to pursue:
\begin{enumerate}
  \item It would be interesting to investigate if the limit cycles in the regularized model persist under singular perturbation limits $\kappa_\np , c_\ns \to \infty$.  Such an observation would help bridge the gap between this perspective and the hybrid systems approach.
  \item While the limit cycle in the paper is stable, the size of the stability basin is not addressed.
  Having a large stability basin is one method of achieving robustness, and so a lower bound for the radius of this basin would be useful to have.
  \item A non-flat ground breaks symmetry, but may still be addressed using normal hyperbolicity theory if the ground is still sufficiently close to flat.
  Similarly, small random or time dependent perturbations will only slightly perturb the relative limit cycle; in particular, the phase shift of each cycle is close to that without these perturbations.
\end{enumerate}

Lastly, we would hope that at least a portion of these ideas would aid in studying stable walking models.
In our model we constructed a crawling-like limit cycle as a small perturbation of an unactuated system and made use of the fact that stability along all of the limit cycle was preserved.
This makes our model not directly applicable to walking, which is typically considered to be `statically unstable' (e.g.\ in the inverted pendulum models, the walker collapses to the floor when the joints are not active).
On the other hand, if one finds a model for walking with a limit cycle that is stable as a whole (that is, its Poincar\'e map is stable), then that cycle can be used as a starting point, and Lie theory can still be used to find a reduced description and a reconstruction formula for the phase shift.
Furthermore, the resulting limit cycles would persists under small perturbations as described above.

\subsection{Acknowledgments}
  The notion of realizing the no-slip condition as a limit of viscous friction was brought to the attention of H.J.\ by Dmitry V.\ Zenkov, while J.E.\ learned this from Hans Duistermaat.
  Sam Burden first enlightened H.J.\ on the role of limit cycles in model reduction for hybrid systems.
  We also thank Tony Bloch, Hamed Rasavi, Justin Seipel, and Ram Vasudevan for helpful conversations during the development of this paper.
  Finally, the initial stimulus to write this paper was given by Jair Koiller, who has been very supportive of our foray into biomechanics.
  Both authors were supported by the European Research Council Advanced Grant 267382 FCCA and H.J.\ also by the NSF grant CCF-1011944.

\appendix

\section{Friction dominated dynamics as singular perturbation}
\label{app:friction-dynamics}

In this appendix we expand a bit more on obtaining first order
equations of motion in the friction dominated regime, i.e.\ when
inertial forces are negligible. We shall rigorously justify the
resulting equations by a geometric singular perturbation argument.
Furthermore, we investigate what happens when one moves away from this
friction dominated limit. For an introduction to geometric singular
perturbation theory we refer the reader to~\cite{Jones1995,Kaper1999}
or the foundational work~\cite{Fenichel1979}.

We consider again a general mechanical system as in
Section~\ref{sec:connections}, without the requirement that $Q$ is a
principal $G$-bundle; we do require that $Q$ is compact\footnote{%
  This is for technical reasons of applying normal hyperbolicity
  theory. Compactness can be replaced by uniformity conditions,
  see~\cite{Eldering2013}.%
}. That is, as in section~\ref{sec:connections} we have a Lagrangian
\begin{equation*}
  L(q,\dot{q}) = \frac{m}{2} k_q(\dot{q},\dot{q}) - V(q)
\end{equation*}
and a Rayleigh dissipation function
\begin{equation*}
  R(q,\dot{q}) = \frac{c}{2} \nu_q(\dot{q},\dot{q})
\end{equation*}
such that both $k$ and $\nu$ are Riemannian metrics on $Q$.
Furthermore, we add a time dependent arbitrary force, which can be
used to control the system, and we absorb the potential term
$-\d V(q)$ into it. This leads to equations of motion
\begin{equation}\label{eq:EoM-forced}
  m\,k^\flat \cdot \nabla^k_{\dot{q}} \dot{q}
  = -c\,\nu_q^\flat \cdot \dot{q} + F(q,t).
\end{equation}
From this, one can formally obtain first order dynamics by setting
$m = 0$. This defines an invariant\footnote{%
  The manifold $M$ is time dependent since the vector field is. This
  seems a contradictory statement, but should be interpreted as $M$
  being invariant in the extended phase space $\T Q \times \R$. That
  is, a solution curve starting in $M(t_0)$ at time $t_0$ ends up in
  $M(t)$ under the time dependent flow $\Phi^{t_0,t}$. We shall
  suppress this explicit time dependence to not clutter the equations too
  much.%
} manifold $M \subset \T Q$, see~\eqref{eq:invar-1st-order-dynamics},
which can also be interpreted as a (time dependent) vector field on
$Q$ with dynamics
\begin{equation}\label{eq:h-section}
  \dot{q} = h(q,t) := \frac{1}{c} \nu_q^\sharp \cdot F(q,t).
\end{equation}
This result can be obtained rigorously by viewing it as a singular
perturbation problem in the limit $m \to 0$. Moreover, the singular
perturbation analysis will allow us to obtain correction terms to the
dynamics for $m$ close to zero; these terms will not be interpretable
anymore as a linear connection on $Q$.

Let us start by writing out~\eqref{eq:EoM-forced} in induced
coordinates on $\T Q$ and rewrite it as a second order system
\begin{equation*}\label{eq:EoM-coords}
  \begin{aligned}
    m \Big(\dot{v}^i + \Gamma^i_{kl}(q) v^k v^l\Big)
    &= - c\, k^{ij}(q)\, \nu_{jk}(q)\, v^k + k^{ij}(q)\, F_j(q,t), \\
    \dot{q}^i &= v^i.
  \end{aligned}
\end{equation*}
The limit $m \to 0$ is singular as $m$ multiplies a
derivative on the left-hand side. This can be remedied by introducing
a rescaled, `fast' time variable $\tau = \frac{t}{m}$, i.e.\ $\tau$
measures time at a fine-grained scale, hence in this time-scale one
mainly observes fast processes. We conventionally denote a derivative
with respect to $\tau$ by a prime and obtain
\begin{equation}\label{eq:EoM-rescaled}
  \begin{aligned}
    v'^i &= - m\,\Gamma^i_{kl}(q) v^k v^l
            - c\, k^{ij}(q)\, \nu_{jk}(q)\, v^k + k^{ij}(q)\, F_j(q,t(\tau)), \\
    q'^i &= m\,v^i.
  \end{aligned}
\end{equation}
Note that this system is well-defined even for $m=0$ and this limit is
aptly called the `frozen time picture' as motion in $q$ has been
killed by the rescaling\footnote{%
  The time dependent term $F(q,t(\tau))$ can still be interpreted
  correctly for $m = 0$, by viewing~\eqref{eq:EoM-rescaled} as a
  rescaling of the vector field without explicitly reparametrizing
  time. See also~\cite[Sect.~4.1]{Eldering2013} for the fact that
  normal hyperbolicity can be extended to this setting; this we will
  use later.%
}. We shall denote by $X_m$ the vector field on $\T Q$ associated
to~\eqref{eq:EoM-rescaled}. The vector field $X_0$ by construction has
\begin{equation*}
  M = \{ c\,\nu_q^\flat \cdot \dot{q} = F(q,t) \} = \text{Graph}(h)
\end{equation*}
as an invariant manifold consisting of fixed points. Furthermore, a
linearization of $X_0$ at points $(q,v) \in M$ along the fiber
direction yields that
\begin{equation*}
  \pder{X_0^i}{v^k}\Big|_M = - c\, k^{ij}(q)\, \nu_{jk}(q).
\end{equation*}
Note that this has strictly negative eigenvalues since
$k^\sharp(q) \cdot \nu^\flat(q)$ is similar to the positive definite
$k^\sharp(q)^\frac12 \cdot \nu^\flat(q) \cdot k^\sharp(q)^\frac12$.
This implies that $M$ is an (attractive) normally hyperbolic invariant
manifold for $X_0$ and hence it persists for sufficiently small
$m > 0$ as a manifold $M_m$ that is invariant under $X_m$ and
diffeomorphic and $C^k$-close to the original $M$ (with
$k \in \mathbb{N}$ large, but depending on $m$),
see~\cite[Thm~1]{Fenichel1971} and~\cite[Thm~4.4]{Hirsch77}. It
follows as an easy corollary that $M_m$ depends $C^k$-smoothly on $m$, see e.g.~\cite[Sect.~4.2]{Eldering2013}.

\begin{figure}[h]
  \centering
  \input{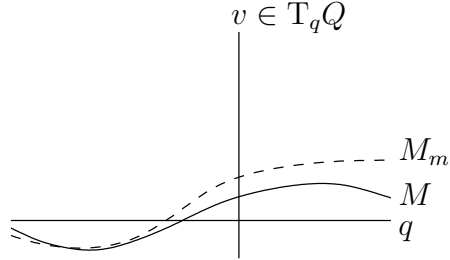}
  \caption{The invariant manifolds $M$ and $M_m$ as graphs of $h$ and $h+\eta_m$.}
  \label{fig:invar-mfld}
\end{figure}

Since $M_m$ is still invariant, we can consider the restricted vector
field $X_m|_{M_m}$ and study its Taylor expansion around $m = 0$. We
may assume that all manifolds $M_m$ are the graph of a section of
$\T Q$, and hence we can represent $X_m|_{M_m}$ by its projection onto
$M$ or onto the base $Q$, which are both fixed.
The latter representation can be identified with the first
order vector field. We calculate the Taylor expansion in local
coordinates adapted to the projection onto $M$, that is, we use
coordinates $q$ on $Q$ and shifted velocity coordinates
$w = v - h(q,t(\tau))$ where $h(q,t)$ is the section that
defines $M$ in $(q,v)$ coordinates, see~\eqref{eq:h-section} and also
Figure~\ref{fig:invar-mfld}. Furthermore, let $w = \eta_m(q,\tau)$ define
$M_m$ as a graph relative to $M$ and note that $\eta_0 \equiv 0$. We introduce
some new notation to shorten the following exposition of the singular
perturbation analysis; this is also more in line with the notation
used in this field. The vector field $X_m$ can be expressed in $(q,w)$
coordinates as
\begin{equation}\label{eq:sing-pert-vf}
  \begin{split}
    q' &= f_m(q,w) = m (h(q,t(\tau)) + w),\\
    w' &= g_m(q,w) = \begin{aligned}[t]
                       -&m\,(\Gamma^i_{kl}(q) v^k v^l \partial_i)
                        - c\,k^\sharp\,\nu^\flat\,v + k^\sharp\,F(q,t(\tau))\\
                       -&m\,\D_i h(q)\, v^i - m\,\partial_t h(q,t(\tau)),
                     \end{aligned}
  \end{split}
\end{equation}
where $v = h(q,t(\tau)) + w$. Furthermore, we expand
\begin{equation*}
  f_m(q,w) = \sum_{k \ge 0} m^k\,f^{(k)}(q,w),
\end{equation*}
implicitly truncated at an appropriate order and we use the same
notation for $g$ and $\eta$.

Invariance of $M_m$ under $X_m$ implies that
\begin{equation*}\label{eq:invariance}
  w' = \frac{\d}{\d \tau} \eta_m(q) = \D \eta_m(q) \, q'.
\end{equation*}
Inserting the vector field~\eqref{eq:sing-pert-vf} into this equation
together with $w = \eta_m(q)$, yields
\begin{equation}\label{eq:invariance-expanded}
  g_m(q,\eta_m(q)) = \D \eta_m(q) \, f_m(q,\eta_m(q)).
\end{equation}
Now we perform a Taylor expansion with respect to $m$ on both sides.
Noting that $\eta_0(q) = 0$, $f_0(q,w) = 0$ and $g_0(q,0) = 0$, we
find $0 = 0$ at zeroth order, and at first order
\begin{equation*}\label{eq:invariance-m1}
  g^{(1)}(q,0) + \D_2 g^{(0)}(q,0) \eta^{(1)}(q) = 0.
\end{equation*}
Normal hyperbolicity of $M$ implies that all eigenvalues of
$\D_2 g^{(0)}(q,0)$ have non-zero (and in our case negative) real
part. Hence we can invert it to solve for $\eta^{(1)}$ and find
\begin{equation*}\label{eq:eta1}
  \eta^{(1)}(q) = \frac{1}{c} \nu^\sharp\, k^\flat \Big[
                    (\nabla^k_h h)(q) +\frac{1}{c} \nu^\sharp\, \partial_t F(q,t(\tau))\Big].
\end{equation*}
Furthermore, the projection of $X_m|_{M_m}$ onto $Q$ is given at first
order by
\begin{equation*}
  q' = m\,f^{(1)}(q,0) + \mathcal{O}(m^2).
\end{equation*}
Note that this (rescaled time) vector field does not contain a zeroth
order term, so we can scale it back to normal time by dividing by $m$.
Letting $\pi\colon \T Q \to Q$ denote the tangent bundle projection,
we thus obtain a well-defined limit vector field
\begin{equation*}
  \lim_{m \to 0} \frac{1}{m} \T \pi \circ X_m \circ (h+\eta_m) \in \mathfrak{X}(Q)
\end{equation*}
which is given in coordinates by
\begin{equation*}\label{eq:normal-vf-limit}
  \dot{q} = f^{(1)}(q,0) = h(q,t),
\end{equation*}
by inserting~\eqref{eq:sing-pert-vf}. Note that this is indeed the first
order dynamics found earlier.

Secondly, we can use the singular perturbation analysis to obtain more
terms in the Taylor expansion of $\T \pi \circ X_m|_{M_m}$, which add
corrections when $m > 0$. These can be found iteratively from the
`master equation'~\eqref{eq:invariance-expanded}; we shall recover one
more term here. A straightforward calculation yields that the second
order term in $f_m(q,\eta_m(q))$ is
\begin{multline*}
  \begin{aligned}
    m^2\,\Big[& \frac{1}{2} f^{(2)}(q,0) + \D_2 f^{(1)}(q,0)\,\eta^{(1)}(q)\\
              &+\frac{1}{2} \D_2^2 f^{(0)}(q,0)\,\eta^{(1)}(q)^2
               +\frac{1}{2} \D_2   f^{(0)}(q,0)\,\eta^{(2)}(q) \Big]
  \end{aligned}\\
  = m^2\, \D_2 f^{(1)}(q,0)\,\eta^{(1)}(q)
  = \frac{m^2}{c} \nu^\sharp\, k^\flat \Big[
      (\nabla^k_h h)(q) + \frac{1}{c} \nu^\sharp\, \partial_t F(q,t(\tau))\Big].
\end{multline*}
This leads to a corrected first order vector field
\begin{equation}\label{eq:normal-vf-m1}
  \dot{q} = h(q,t) + \frac{m}{c}  \nu^\sharp\,k^\flat \Big[
                       (\nabla^k_h h)(q) + \frac{1}{c} \nu^\sharp\,\partial_t F(q,t)
                     \Big]
                   + \mathcal{O}(m^2).
\end{equation}
Note that the term in brackets could be interpreted as the total time
derivative of $h(q,t)$, were it not that this would introduce a
circular dependency in the definition of $\dot{q}$.

Finally, let us return to the context of $Q$ being a left
$G$-principal bundle. The (ideal) Stokesian regime can be defined as
the values of $m$ and $c$ where
\begin{equation*}
  \dot{q} = h(q,t) = \frac{1}{c} \nu^\sharp F(q,t)
\end{equation*}
holds accurately. Let us assume that $F(q,t)$ is a control force that
acts on the shape space $S = G \backslash Q$. This means that $F$
takes values in the annihilator of $\textrm{Ver}(\T Q)$, i.e.\ $F$
does no work along displacements along $G$-orbits. Then we can view
$u(q,t) := \nu^\sharp F(q,t) \in \textrm{Hor}^\nu(\T Q)$ as a control
on shape space, and the Stokes connection determines how solution
curves $s(t) \in S$ are lifted to curves in $G \backslash \T Q$.

However, if we extend our notion of the Stokesian regime and include
the first order perturbation terms in~\eqref{eq:normal-vf-m1}, then
the vector field generally does not take values in
$\textrm{Hor}^\nu(\T Q)$ anymore. This is because
$\nu^\sharp\,k^\flat$ does not preserve this subbundle and we can
choose $h(\,\cdot\,,t) = 0$ and $\partial_t F(\,\cdot\,,t) \neq 0$
independently such that $\nabla^k_h h = 0$ while
$\nu^\sharp\,\partial_t F(q,t) \in \textrm{Hor}^\nu(\T Q)$. Thus, in
this perturbed Stokes regime, the well-known Scallop Theorem does not
hold anymore. This agrees with a numerical experiment we performed
where the shape force curve $F(\,\cdot\,,t)$ had one-dimensional
image, but non-constant time parametrization and a small, non-zero
phase shift was observed. We conclude that our crawler model seems to be in
the `perturbed Stokes regime' but not in the Stokes regime in the
classical sense.

\section{Stability proofs}
\label{app:stability-proofs}

In this appendix we collect the detailed proofs for the statements in Section~\ref{sec:stable-equilibria}.

\begin{proof}[Proof of Proposition~\ref{prop:potential-minimum}]
  To simplify the analysis we change to a (local) coordinate system for $\SE(2) \backslash Q$ given by $(\ell,Z)$ with $\ell = (\ell_{12},\dots,\ell_{34}) \in (\mathbb{R}^+)^6$ and $Z = (z_1,z_2,z_3)$.
  In these coordinates, and under the assumption that the height of the $4$th mass is positive, the (reduced) potential energy takes the form
  \begin{equation*}\label{eq:red-potential}
    \hat{U} = \left( \frac{\kappa_\s}{2} \sum_{j>i} ( \ell_{ij} - \bar{\ell}_{ij})^2  \right)
             + \left( \sum_{i=1}^{3} z_i + \kappa_{\np}\, \chi(z_i) \right) + z_4(\ell,Z).
  \end{equation*}
  Note that $z_4$, the gravitational potential of the $4$th mass, depends on the shape variables $\ell$ and $Z$ in an intricate way which we shall not endeavor to make explicit.
  Thus we search for a solution $s_* = (\ell_*,Z_*) \in \SE(2) \backslash Q$ of
  \begin{equation}\label{eq:dU-zero}
    \begin{aligned}
      0 = \d\hat{U}(s_*)
       &= \sum_{j>i} \d\ell_{ij} \left( \kappa_\s (\ell_{ij} - \bar{\ell}_{ij}) + \pder{z_4}{\ell_{ij}} \right)\\
       &{}\;\;+\sum_{i=1}^3 \d z_i \left( 1 + \kappa_{\np}\,\chi'(z_i) + \pder{z_4}{z_i} \right).
    \end{aligned}
  \end{equation}
  We recover the solution $s_*$ by an implicit function argument.
  Let us define the function
  \begin{equation*}\label{eq:impl-func}
    F\big((\ell,Z),\varepsilon\big) =
    \begin{bmatrix}
      \ell_{12} - \bar{\ell}_{12} + \varepsilon \pder{z_4}{\ell_{12}} \\
      \vdots \\
      \ell_{34} - \bar{\ell}_{34} + \varepsilon \pder{z_4}{\ell_{34}} \\[0.5em]
      1 + \kappa_{\np}\,\chi'(z_1) + \pder{z_4}{z_1} \\
      \vdots \\
      1 + \kappa_{\np}\,\chi'(z_3) + \pder{z_4}{z_3}
    \end{bmatrix} \in \R^9.
  \end{equation*}
  A zero of $F$ corresponds to a solution of~\eqref{eq:dU-zero} if we set the parameter $\varepsilon = 1 / \kappa_\s$;
 we first search for a zero with $\varepsilon = 0$ though.
 That is, we consider the singular limit of infinite spring stiffness.
 This implies $\ell_{ij} = \bar{\ell}_{ij}$.
 Note that when the ground potential $\kappa_{\np}\chi$ rises steeply enough, it follows by energy arguments that $z_1 \approx z_2 \approx z_3 \approx 0$, so the springs $\ell_{12},\ell_{13},$ and $\ell_{23}$ are oriented approximately horizontally.

 Now we shall use a geometric argument to show that $1+\partial z_4 / \partial z_i > 0$ for $i = 1,2,3$ with $\ell = \bar{\ell}$ fixed.
 First, w.l.o.g.\ we can assume that the rigid tetrahedron with lengths $\bar{\ell}_{ij}$ and with masses $1,2,3$ on the ground is in stable equilibrium, possibly by permuting the masses.
 An equilibrium exists by potential energy minimization, and this minimum must be non-degenerate; if it were not, then the center of mass would be above one of the ground edges, but rotation about this axis would then lower the center of mass, see Figure~\ref{fig:stable_tetrahedron}.
 This image also shows that mass 2 must be closer to the edge $\bar{\ell}_{13}$ horizontally than mass 4, which implies that $\partial z_4 / \partial z_2 > -1$. The same holds for $i=1,3$ too.

 \begin{figure}[htb]
   \centering
   \includegraphics{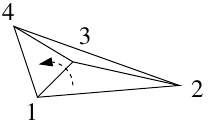}
   \caption{A tetrahedron in stable equilibrium with masses $1,2,3$ on the ground plane and a possible rotation about the edge $\bar{\ell}_{13}$.}
   \label{fig:stable_tetrahedron}
 \end{figure}

  Further, $\chi'(z)$ is monotonically decreasing without bound from $0$ as $z \to -\infty$.  It follows that there are unique values $z_i < 0$ such that the point $s_0 = (\bar{\ell}_{12},\dots, \bar{\ell}_{34},z_1,z_2,z_3 )$ solves $F(s_0,0) = 0$.

  The derivative of $F$ with respect to the variables $(\ell,Z)$ at $s_0$ is found to be
  \begin{equation*}
    DF(s_0,0) =
    \begin{bmatrix}
      I_6 & B \\
      0    & A+\kappa_{\np}\,I_3
    \end{bmatrix},
  \end{equation*}
  where $B_{ij,k} =\pder{^2 z_4}{\ell_{ij} \partial z_k}$ for $j>i$ and $A$ is the Hessian of $Z \mapsto z_4(\bar{\ell},Z)$.
 Note that if $\kappa_{\np}$ is sufficiently large, then $A+\kappa_{\np}\,I_3$ is positive definite.
 The eigenvalues $\lambda$ of $DF(s_0,0)$ are recovered from
  \begin{equation*}
    0 = \det(DF(s_0,0) -  \lambda I_9)
      = (1-\lambda)^3 \det(A+\kappa_{\np}\,I_3 - \lambda I_3)
  \end{equation*}
  and found to be all positive.
  In particular $DF(s_0,0)$ is invertible and we can apply the implicit function theorem to conclude that there exists an $\varepsilon_0 > 0$ such that for any $0 \le \varepsilon < \varepsilon_0$ there exists a $s_\varepsilon$ such that $F(s_\varepsilon,\varepsilon) = 0$.
  Setting $\kappa_\s = 1/\varepsilon$ will give that $s_* = s_\varepsilon$ is a solution for~\eqref{eq:dU-zero}.

  Before fixing $\varepsilon$, let us prove that the Hessian $\hat{\kappa}$ of $\hat{U}$ at a candidate minimizer $s_\varepsilon$ is positive definite.
  From the definition of the potential it follows that
  \begin{equation*}\label{eq:hatK}
    \hat{\kappa} =
    \begin{bmatrix}
      \kappa_{\np}\,I_3 & 0\\
      0                 & \kappa_\s\,I_6
    \end{bmatrix} + D^2 z_4,
  \end{equation*}
  where $D^2 z_4$ is the Hessian of $z_4$ as a function of $\ell$ and $Z$.
  Note that the first term is positive definite and by choosing $\kappa_\s$ and $\kappa_{\np}$ sufficiently large, we can make it dominate the term $D^2 z_4$ such that $\hat{\kappa}$ as a whole is positive definite.
  We finally choose $\varepsilon$ sufficiently small such that we obtain both that $s_* = s_\varepsilon$ is a minimizer of $\hat{U}$ and $\kappa_\s = 1/\varepsilon$ is large enough that $\hat{\kappa}$ is positive definite.
\end{proof}

To prove Proposition~\ref{prop:nondegenerate} we invoke the following Lemma.

\begin{lem}\label{lem:positive_def}
  If $A_1, \dots , A_n$ are positive semi-definite linear operators on a finite-dimensional inner-product space $(V , \lb \cdot , \cdot \rb)$ and $\bigcap_{k = 1}^{n}{ \ker(A_k) } = \{ 0 \}$,
  then $A = \sum_{k=1}^{n}{ A_k }$ is positive definite.
\end{lem}

\begin{proof}
  Clearly $A$ is positive semi-definite as a sum of semi-definite operators.
  We must prove that $A$ is definite.
  Assume $A$ is not definite so that there exists some non-zero $x \in V$ such that $\lb x , Ax \rb = 0$.
  This latter equation can be written as $\sum_{k=1}^{n}{ \lb x , A_k x \rb } = 0$.
  By semi-definiteness of each $A_k$ this implies $\lb x , A_k x \rb = 0$.
  This means that $A_k x = 0$ for each $k$.
  However the only such $x$ is $0$.
\end{proof}

\begin{proof}[Proof of Proposition~\ref{prop:nondegenerate}]
	Let $q_* \in Q$ be such that masses $1$, $2$ and $3$ are within the influence of the ground forces (i.e.\ such that the $z$ coordinates are within the support of $\chi$).
	The force $F_\s:TQ \to \T^*Q$ can be expressed as a degenerate metric $\nu_\s : TQ \oplus TQ \to \mathbb{R}$
	via the equation $\nu_\s(v,w) =\langle F_\s(v) , w  \rangle$ where $\langle \cdot , \cdot \rangle$ denotes the
	canonical pairing between $\T^*Q$ and $TQ$.
	The same can be said of forces $F_{\ns}$ and $F_{\db}$ with respect to degenerate metrics $\nu_{\ns}$ and $\nu_{\db}$.
	
	We can see that $\nu_\s = c_\s \sum_{i < j} \d \ell_{ij} \otimes \d \ell_{ij}$.
	Thus the kernel of $\nu_\s$ is the set of infinitesimal transformations which preserve the lengths of the spring line segments.
	By assumption the springs form a non-degenerate tetrahedron, so these transformations are generated by $\mathfrak{se}(3)$, the $6$-dimensional space of infinitesimal isometries of $\R^3$.
	We can denote the generated space by $\mathfrak{se}(3) \cdot q_*$.

	Under standing the assumption that $z_4 > 0$, we find that
	\begin{equation*}
	  \nu_{\ns} = c_{\ns} \sum_{i=1}^{3} \chi'(z_i) (\d x_i \otimes \d x_i + \d y_i \otimes \d y_i),
	\end{equation*}
	so the kernel is precisely spanned by the infinitesimal changes in height of masses $1$, $2$ and $3$,
	as well as arbitrary infinitesimal changes in position of mass $4$.
	That is, translations along the coordinate directions $z_1,\dots,z_4$ as well as $x_4$ and $y_4$.
	
	Finally, $\nu_{\db} =  c_{\db} \sum_{i=1}^{3} \chi(z_i) \d z_i \otimes \d z_i$
	so its kernel consists of translations along the coordinate directions $x_1,\dots,x_4,y_1\dots,y_4,$ and $z_4$.
	
	We then observe directly that $\ker(\nu_{\db}) \cap \ker(\nu_{\ns}) = {\rm span} \left( \partial_{x_4},\partial_{y_4},\partial_{z_4}\right)$.
	Such transformations will move the $4$th mass, while keeping the others fixed.
	This is not a rigid transformation generated by $\mathfrak{se}(3)$.
	Therefore
	\begin{align*}
		\ker(\nu_{\db}) \cap \ker(\nu_{\ns}) \cap \ker(\nu_\s) = {\rm span} \left( \partial_{x_4},\partial_{y_4},\partial_{z_4}\right) \cap \mathfrak{se}(3) \cdot q_* = \{0 \}.
	\end{align*}
	By Lemma \ref{lem:positive_def} then, $\nu = \nu_{\db} + \nu_{\ns} + \nu_\s$ is positive definite on the fiber above $q_*$.
	As $\hat{\nu}$ is merely the push-forward of $\nu$ by the projection $\Pi: TQ \to P$, it is related to $\nu$ by an outer automorphism
	and is therefore positive definite as well.
\end{proof}

\begin{proof}[Proof of Proposition~\ref{prop:stability}]
  Firstly, $(s_*,0)$ is an equilibrium for the reduced system, and its linearization is given by Proposition~\ref{prop:linearization}.
  To assert that it is a robustly stable equilibrium, we consider its linearization~\eqref{eq:linear-reduced},
  \begin{equation*}
    \frac{\d}{\d t}\begin{bmatrix} s \\ \xi \end{bmatrix}
    = A            \begin{bmatrix} s \\ \xi \end{bmatrix}
    \quad\text{with}\quad
    A = \begin{bmatrix} 0 & \pr \\ - \kappa\,\pr^T & -\hat{\nu} \end{bmatrix},
  \end{equation*}
  where $\pr = \begin{bmatrix} I_9 & 0 \end{bmatrix}$ represents the principal bundle projection $\pi\colon Q \to \SE(2) \backslash Q$ in fiber-adapted coordinates.
  Recall that $\kappa$ and $\hat{\nu}$ are positive (semi-)definite matrices describing the linearized potential and friction forces, respectively.
  It follows from the definition $\widehat{U} = U \circ \pr$ and $\pr \, \pr^T = I_5$ that $\kappa\,\pr^T = \pr^T \hat{\kappa}$.

  Note that it is sufficient to prove that the linear flow satisfies $\norm{e^{At_0}} \le r < 1$ for some $t_0 > 0$, $r < 1$, and any choice of norm.
  From this it follows that the flow contracts exponentially for large $t$: write $t = n t_0 + \tau$ with $n \in \mathbb{N}$ and $\tau \in [0,t_0)$, then we have
  \begin{equation*}
      \norm{e^{A t}}
    = \norm{e^{A(n t_0 + \tau)}}
    = \norm{(e^{A t_0})^n e^{A\tau}}
    \le \sup_{0 \le \tau \le t_0} \norm{e^{A\tau}} r^n
    = C e^{\rho t}
  \end{equation*}
  with $\rho = \frac{\log(r)}{t_0} < 0$ and $C = \sup_{0 \le \tau \le t_0} \norm{e^{A\tau}}e^{-\rho \tau} < \infty$.

  We choose the norm induced by the (approximate) energy function
  \begin{equation*}\label{eq:linear-hamiltonian}
    E_L(s,\xi) = \frac{1}{2}\lb \xi,\xi \rb
                +\frac{1}{2}\lb s,\hat{\kappa}\,s \rb
  \end{equation*}
  for the linear system~\eqref{eq:linear-reduced}, i.e.\ $E_L = \norm{\,\cdot\,}^2$.
  This energy is a (non-strict) Lyapunov function in the sense that
  \begin{equation*}
    \frac{\d E_L}{\d t}
    = \pder{E_L}{s}   \frac{\d s}{\d t}
     +\pder{E_L}{\xi} \frac{\d \xi}{\d t}
    = \lb \hat{\kappa} \cdot s , \pr \cdot \xi \rb
     +\lb \xi , -\kappa\,\pr^T \cdot s - \hat{\nu} \cdot \xi \rb
    = -\lb \xi ,\hat{\nu} \cdot \xi \rb < 0
  \end{equation*}
  for all $\xi \neq 0$, since $\hat{\nu}$ is positive definite.
  To prove that $\norm{e^{At_0}} \leq r < 1$, let $\norm{(s,\xi)} = 1$ and note that since $E_L$ is non-increasing along solution curves, we can from now on restrict our analysis to the compact ball $\overline{B(0;1)} = E_L^{-1} ( [0,1] )$.

  The proof would be finished if $E_L$ were strictly decreasing, but this does not hold true for points $(s,0)$ in phase space.
  Instead, then, we have $\dot{\xi} = -\pr^T \hat{\kappa} s \neq 0$, so after a short time interval, $\xi \neq 0$, and thus $E_L$ starts decreasing.
  Thus fixing a $t_0 > 0$, we find that $E_L$ strictly decreases along any solution curve over a time interval of length $t_0$, for all initial conditions $\norm{(s,\xi)} = 1$.
  By continuous dependence of a flow on initial parameters and compactness, it follows that the decrease of $E_L$ is uniformly bounded away from zero, and hence we have $\norm{e^{At_0}} \le r < 1$ for some $r < 1$.
\end{proof}

\bibliographystyle{amsplain}
\bibliography{hoj_je_crawler}

\end{document}